\documentclass[11pt,oneside]{amsart}
\usepackage[top=25mm, bottom=25mm, left=20mm, right=20mm]{geometry}
\usepackage{amsfonts,amssymb,amsmath,amsthm,amsxtra}
\usepackage[T1]{fontenc}
\usepackage[utf8]{inputenc}
\usepackage{bm}
\usepackage{mathtools}
\usepackage{dsfont}
\usepackage{bbm}
\usepackage{mathrsfs}
\usepackage{enumitem}
\usepackage{color}

\usepackage{graphics}

\numberwithin{equation}{section}
\newtheorem{theorem}{Theorem}
\newtheorem{corollary}[theorem]{Corollary}
\newtheorem{lemma}[theorem]{Lemma}
\newtheorem{proposition}[theorem]{Proposition}
\theoremstyle{definition}
\newtheorem{remark}[equation]{Remark}

\newcommand{\CC}{\mathbb{C}}

\newcommand{\GG}{\mathbb{G}}

\newcommand{\II}{\mathbb{I}}

\newcommand{\LL}{\mathbb{L}}

\newcommand{\NN}{\mathbb{N}}

\newcommand{\PP}{\mathbb{P}}
\newcommand{\QQ}{\mathbb{Q}}
\newcommand{\RR}{\mathbb{R}}

\newcommand{\TT}{\mathbb{T}}

\newcommand{\ZZ}{\mathbb{Z}}

\newcommand{\N}{\mathbb{N}}

\newcommand{\calA}{\mathcal{A}}

\newcommand{\calP}{\mathcal{P}}
\newcommand{\calF}{\mathcal{F}}
\newcommand{\calB}{\mathcal{B}}
\newcommand{\calM}{\mathcal{M}}

\newcommand{\calO}{\mathcal{O}}
\newcommand{\calH}{\mathcal{H}}

\newcommand{\calK}{\mathcal{K}}

\newcommand{\calS}{\mathcal{S}}

\newcommand{\calR}{\mathcal{R}}

\newcommand{\calQ}{\mathcal{Q}}

\newcommand{\frakm}{\mathfrak{m}}

\newcommand{\frkS}{\mathfrak{S}}

\newcommand{\frkM}{\mathfrak{m}}
\newcommand{\frkN}{\mathfrak{n}}
\newcommand{\frkY}{\mathfrak{y}}

\newcommand{\dif}{\mathrm{d}}
\newcommand{\ex}{\bm{e}}

\newcommand{\ind}[1]{\mathds{1}_{{#1}}}

\newcommand*{\DMO}[1]{\expandafter\DeclareMathOperator\csname #1\endcsname {#1}}
\DMO{ord}
\DMO{supp}
\DMO{Sym}
\DMO{cl}
\DMO{lcm}
\DMO{dist}
\DMO{tr}
\DMO{diam}

\DeclarePairedDelimiter\abs{\lvert}{\rvert}

\DeclarePairedDelimiter\norm{\lVert}{\rVert}

\DeclarePairedDelimiterX\spr[2]{\langle}{\rangle}{#1,#2}
\newcommand{\ipr}[2]{#1\cdot#2}
\DeclarePairedDelimiterX\Set[2]{\{}{\}}{#1\colon #2}
\DeclarePairedDelimiterX\Seq[1]{(}{)}{#1}

\begin{document}
\title[Estimates for ergodic averages along multi-dimensional subsets of primes]{Oscillation and jump inequalities for the polynomial ergodic averages along multi-dimensional subsets of primes}

\author{Nathan Mehlhop \and Wojciech S{\l}omian}

\address[Nathan Mehlhop]{Department of Mathematics,
Rutgers University,
Piscataway, NJ 08854-8019, USA }
\email{nam225@math.rutgers.edu}

\address[Wojciech S{\l}omian]{BCAM - Basque Center for Applied Mathematics, 48009 Bilbao, Spain \&
	  Faculty of Pure and Applied Mathematics, 
	  Wroc{\l}aw University of Science and Technology\\
	  Wyb{.} Wyspia\'nskiego 27,
	  50-370 Wroc\l{}aw, Poland}
    \email{wojciech.slomian@pwr.edu.pl\ \& wslomian@bcamath.org}
\subjclass[2020]{37A30 (Primary), 37A46, 42B20}
\keywords{oscillation seminorm, jump inequality, ergodic average along primes}
\thanks{The authors were partially supported by the National Science Foundation (NSF) grant DMS-2154712. The second author was supported by the Basque Government
through the BERC  2022--2025 program and by the Ministry of
Science, Innovation and Universities: BCAM Severo Ochoa
accreditation SEV-2017-0718.}

\begin{abstract}
We prove the uniform oscillation and jump inequalities for the polynomial ergodic averages modeled over multi-dimensional subsets of primes. This is a contribution to the Rosenblatt-Wierdl conjecture \cite[Problem 4.12, p. 80]{RW} with averages taken over primes. These inequalities provide endpoints for the $r$-variational estimates obtained by Trojan \cite{Troj}.
\end{abstract}
\maketitle
\section{Introduction}
The aim of this paper is to prove uniform oscillation inequalities and $\lambda$-jump inequalities in the context of polynomial ergodic averages and truncated singular operators of the Cotlar type modeled on multi-dimensional subset of primes. We extend the known results of Trojan \cite{Troj} for the $r$-variation seminorm $V^r$ with $r>2$ to endpoint cases expressed in terms of the uniform jump and oscillation inequalities. This provides a fuller quantitative description of the pointwise convergence of the mentioned averages.
\subsection{Statement of results}
 Let $(X,\calB,\mu)$ be a  $\sigma$-finite measure space endowed with a family of invertible commuting and measure preserving transformations $S_1,\ldots, S_d:X\to X$. Let $\Omega$ be a bounded convex open subset of $\RR^k$ such that $B(0,c_{\Omega}) \subseteq \Omega \subseteq B(0,1)$ for some $c_{\Omega}\in(0, 1)$, where $B(0, u)$ is the open Euclidean ball in $\RR^k$ with radius $u>0$ centered at $0\in\RR^k$. For any $t>0$, we set
\[
\Omega_t := \{x\in\RR^{k}: t^{-1}x\in\Omega\}.
\] 
We consider a polynomial mapping
\begin{equation}\label{polymap}
    \mathcal{P}=(\mathcal{P}_1,\dots,\mathcal{P}_{d})\colon\ZZ^k\to\ZZ^{d}
\end{equation}
where each $\calP_j\colon\ZZ^k\to\ZZ$ is a polynomial of $k$ variables with integer coefficients such that $\calP_j(0)=0$. Let $k',k''\in\{0,1,\ldots,k\}$ with $k=k'+k''$. For $f\in L^\infty(X,\mu)$, we define the associated ergodic averages by  
\begin{equation}\label{def:average}
\calA_{t}^{\calP,k',k''}f(x):=\frac{1}{\vartheta_\Omega(t)}\sum_{(n,p)\in\ZZ^{k'}\times(\pm\PP)^{k''}}f(S_1^{\calP_1(n,p)}\cdots S_d^{\calP_d(n,p)}x)\mathds{1}_{\Omega_t}(n,p)\Big(\prod_{i=1}^{k''}\log|p_i|\Big),\quad x\in X,
\end{equation}
where $\pm\PP$ denotes the set of positive and negative prime numbers and
\begin{equation*}
\vartheta_\Omega(t):=\sum_{(n,p)\in\ZZ^{k'}\times(\pm \PP)^{k''}}\mathds{1}_{\Omega_t}(n,p)\Big(\prod_{i=1}^{k''}\log|p_i|\Big)
\end{equation*}
is the Chebyshev function. We also consider the Cotlar type ergodic averages given by
\begin{equation}\label{def:singular}
\calH_{t}^{\calP,k',k''}f(x):=\sum_{(n,p)\in\ZZ^{k'}\times(\pm\PP)^{k''}}f(S_1^{\calP_1(n,p)}\cdots S_d^{\calP_d(n,p)}x)K(n,p)\mathds{1}_{\Omega_t}(n,p)\Big(\prod_{i=1}^{k''}\log|p_i|\Big),\quad x\in X,
\end{equation}
where $K\colon\RR^{k}\setminus\{0\} \to \CC$ is a Calder\'on--Zygmund kernel satisfying the following conditions:
\begin{enumerate}
\item The size condition: For every $x\in\RR^k\setminus\{0\}$, we have
\begin{equation}
\label{eq:size-unif}
\abs{K(x)} \lesssim \abs{x}^{-k}.
\end{equation}
\item The cancellation condition: For every $0<r<R<\infty$, we have
\begin{equation}
\label{eq:cancel}
\int_{\Omega_{R}\setminus \Omega_{r}} K(y) \dif y = 0.
\end{equation}
\item
  The Lipschitz continuity condition:
  For every $x, y\in\RR^k\setminus\{0\}$ with $2|y|\leq |x|$, we have
\begin{equation}
\label{eq:K-modulus-cont}
\abs{K(x)-K(x+y)}
\lesssim 
\abs{y} \abs{x}^{-(k+1)}.
\end{equation}
\end{enumerate}

\indent We recall the definitions of the oscillation seminorm and $\lambda$-jump counting function. Let $\II\subseteq\RR$. For an increasing sequence $I=(I_j: j\in\NN)\subseteq\II$ and $N\in\NN\cup\{\infty\}$, the {\it truncated oscillation seminorm} of a function $f\colon \II \to \CC$ is defined by
\begin{align}
\label{def:osc}
O_{I, N}^2( f(t): t\in \II)
:= \Big(\sum_{j=1}^N\sup_{\substack{I_j \le t < I_{j+1}\\t\in\II}}
\abs{f(t)-f(I_j)}^2\Big)^{1/2}.
\end{align}
For any $\lambda>0$ and $\II \subseteq \RR$, the \emph{$\lambda$-jump counting function} of a function $f\colon\II \to \CC$ is defined by
\begin{align}
  \label{def:jump}
N_{\lambda}(f(t) : t\in\II):=\sup \{J\in\NN\, |\, \exists_{\substack{t_{0}<\dotsb<t_{J}\\ t_{j}\in\II}}  : \min_{0<j\leq J} \abs{f(t_{j})-f(t_{j-1})} \geq \lambda\}.
\end{align}
We can now state the main result of this paper.
\begin{theorem}\label{MAINthm}
Let $d, k\ge 1$ and let $\calP$ be a polynomial mapping as in ~\eqref{polymap}. Let $k',k''\in\{0,1,\ldots,k\}$ with $k'+k''=k$ and let $\calM_t^{\calP,k',k''}$ be either $\calA_t^{\calP,k',k''}$ or $\calH_t^{\calP,k',k''}$. Then, for any $p\in(1,\infty)$, there is a constant $C_{p,d,k,\deg\calP}>0$ such that
\begin{align}
    \sup_{\lambda>0}\norm[\big]{\lambda N_{\lambda}(\calM_t^{\calP,k',k''} f:t>0)^{1/2}}_{L^p(X,\mu)}&\le C_{p,d,k,\deg \calP}\|f\|_{L^p(X,\mu)},\label{jumpest}\\
    \sup_{N\in\NN}\sup_{I\in\mathfrak{S}_N(\RR_+)}\norm[\big]{O_{I,N}^2(\calM_t^{\calP,k',k''} f:t>0)}_{L^p(X,\mu)}&\le C_{p,d,k,\deg\calP}\norm{f}_{L^p(X,\mu)},\label{oscillest}
\end{align}
for any $f\in L^p(X,\mu)$. Here, $\mathfrak{S}_N(\RR_+)$ is the set of all strictly increasing sequences in $\RR_+$ of length $N+1$ (see Section~\ref{notation:semi}). 
The constant $C_{p,d,k,\deg\calP}$ is independent of the
coefficients of the polynomial mapping $\calP$.
\end{theorem}

In the proof of the above theorem, we use methods developed in \cite{MST2,MSZ3,Troj} and very recently in \cite{MSS,S}. We follow Bourgain's approach \cite{B3} to use the Calderón transference principle \cite{Cald} which reduce the problem to the integer shift system (see Section~\ref{sec:red}) and then exploit the Hardy--Littlewood circle method to analyze the appropriate Fourier multipliers.
The main tools used to handle the estimates for the multiplier operators are: an appropriate generalization of Weyl's inequality (Proposition~\ref{weylinq}); the Ionescu--Wainger multiplier theorem (see \cite{IW,MSZ3} and \cite{TaoIW}) combined with the Rademacher--Menshov inequality (see \cite{MST2}) and standard multiplier approximations (Lemma~\ref{multiplierapprox}); the Magyar--Stein--Wainger sampling principle \cite{MSW} and \cite{MSZ1}.

As a consequence of Theorem~\ref{MAINthm}, we can state the following quantitative form of the ergodic theorem concerning the averages $\calA_t^{\calP,k',k''}$ and $\calH_t^{\calP,k',k''}$.
\begin{corollary}
\label{thm:cor1}
Let $(X,\calB,\mu)$ be a $\sigma$-finite measure space.
Let $d, k\ge 1$ and let $\calP$ be a polynomial mapping as in ~\eqref{polymap}. Let $k',k''\in\{0,1,\ldots,k\}$ with $k'+k''=k$ and let $\calM_t^{\calP,k',k''}$ be either $\calA_t^{\calP,k',k''}$ or $\calH_t^{\calP,k',k''}$. Let $p\in(1,\infty)$ and $f\in L^p(X,\mu)$. Then we have:
\begin{itemize}
\item[(i)] \textit{(Mean ergodic theorem)} the averages
$\calM_t^{\calP,k',k''}f$ converge in $L^p(X,\mu)$ norm as $t\to \infty$;

\item[(ii)] \textit{(Pointwise ergodic theorem)} the averages
$\calM_t^{\calP,k',k''}f$ converge pointwise $\mu$-almost everywhere on $X$ as $t\to\infty$;

\item[(iii)] \textit{(Maximal ergodic theorem)}
the following maximal estimate holds:
\begin{align}
\label{thm:max}
\big\|\sup_{t>0}|\calM_t^{\calP,k',k''}f|\big\|_{L^p(X, \mu)}\lesssim_{d,k,p, \deg \mathcal P}\|f\|_{L^p(X, \mu)};
\end{align}
\item[(iv)] \textit{(Oscillation ergodic theorem)}
the following uniform oscillation inequality holds:
\begin{align}
\label{thm:osc}
\sup_{N\in\NN}\sup_{I\in\mathfrak S_N(\RR_+) }\big\|O_{I, N}^2(\calM_t^{\calP,k',k''}f: t>0)\big\|_{L^p(X, \mu)}\lesssim_{d,k,p, \deg \mathcal P}\|f\|_{L^p(X, \mu)};
\end{align}
\item[(v)] \textit{(Variational ergodic theorem)}
for any $r\in(2,\infty)$, the following $r$-variational inequality holds (see Section~\ref{notation:semi} for the definition of $V^r$):
\begin{align}
\label{thm:var}
\big\|V^r(\calM_t^{\calP,k',k''}f: t>0)\big\|_{L^p(X, \mu)}\lesssim_{d,k,p, r,  \deg \mathcal P}\|f\|_{L^p(X, \mu)};
\end{align}
\item[(vi)] \textit{(Jump ergodic theorem)}
the following jump inequality holds:
\begin{align}
\label{thm:jump}
\sup_{\lambda>0}\big\|\lambda N_{\lambda}(\calM_t^{\calP,k',k''}f:t>0)^{1/2}\big\|_{L^p(X,\mu)}\lesssim_{d,k,p, \deg \mathcal P}\|f\|_{L^p(X, \mu)}.
\end{align}

\end{itemize}
The implicit constants in \eqref{thm:max}, \eqref{thm:osc}, \eqref{thm:var}, and \eqref{thm:jump} are independent of the
coefficients of the polynomial mapping $\calP$.
\end{corollary}
A few comments are in order.
\begin{enumerate}
    \item Corollary~\ref{thm:cor1} is the most general quantitative version of the one parameter ergodic theorem for both averages $\calA_t^{\calP,k',k''}$ and $\calH_t^{\calP,k',k''}$ (cf. \cite[Theorem 1.20]{MISZWR}), which concludes the work of many authors over last decades -- see Section~\ref{sec:hist} for details.
    \item The mean ergodic theorem \textit{(i)} easily follows from \textit{(ii) }and \textit{(iii)} by Lebesgue's dominated convergence theorem. Each inequality from \textit{(iv)}, \textit{(v)}, and \textit{(vi)} individually implies pointwise convergence \textit{(ii)} and the maximal estimate \textit{(iii)}. The jump inequality \textit{(vi)} implies the variational ergodic theorem \textit{(v)} in the full range $r\in(2,\infty)$. Hence, the inequality \eqref{thm:jump} can be seen as an $r=2$ endpoint for \eqref{thm:var}.
    \item Unfortunately, we do not know at this moment if the oscillation inequality \eqref{thm:osc} is any kind of endpoint for the variational inequality \eqref{thm:var}. A recent result from \cite{MSS} shows that the oscillation estimates cannot be interpreted as an endpoint in a way similar to how the jump inequalities are. See the discussion in \cite{MSS} and \cite{MISZWR}.
    \item The oscillation inequality~\eqref{thm:osc} for the ergodic averages $\calA_t^{\calP,k',k''}$ can be seen as a contribution to a problem posed by Rosenblatt and Wierdl \cite[Problem 4.12, p. 80]{RW} in the early 1990’s about uniform oscillation inequalities for the classical Birkhoff ergodic averages given by
    \begin{equation}\label{eq:birk}
        \frac{1}{2N+1}\sum_{n=-N}^Nf(S^nx).
    \end{equation}
    In 1998, Jones, Kaufman, Rosenblatt, and Wierdl \cite{jkrw} gave an affirmative answer to this problem. The inequality \eqref{oscillest} provides us with the uniform oscillation inequality for the counterpart of \eqref{eq:birk} along the prime numbers given by 
    \begin{equation*}
        \frac{1}{2|\PP_N|}\sum_{n=-N}^N{f(S^nx)\mathds{1}_{\PP}(|n|)},
    \end{equation*}
    where $\PP_N=\PP\cap[1,N]$. Moreover, the inequality \eqref{oscillest} is much more general than the originally posted problem since it concerns the multi-dimensional averages along arbitrary polynomials with integer coefficients.
    \item Parts \textit{(i)}, \textit{(ii)}, \textit{(iii)}, and \textit{(v)} for the standard averages $\calA_t^{\calP,k',k''}$ with $k''\geq1$ in the presented generality were first obtained by Trojan \cite{Troj}. In the case with $k''=0$ (excluding the prime numbers from the summation), the first proof of the variational inequality \eqref{thm:osc} in the full range $r\in(2,\infty)$ was given by Mirek, Stein, and Trojan \cite{MST2}.
    \item In the case of the Cotlar ergodic averages $\calH_t^{\calP,k',k''}$ with $k''\geq1$, the ergodic theorems \textit{(i)}, \textit{(ii)}, \textit{(iii)} and \textit{(v)} were proven by Trojan~\cite{Troj} under the gradient condition 
    \begin{equation}\label{eq:gradient}
        |x|^{k+1}|\nabla K(x)|\lesssim 1,
    \end{equation}
   but Trojan's argument can be adapted with small changes to deal with Calderón--Zygmund kernels which satisfy the more general condition \eqref{eq:K-modulus-cont}. In this case, the results for $k''\geq1$ seem to be completely new. For $\calH_t^{\calP,k,0}$, the jump inequality was obtained by Mirek, Stein, and Zorin-Kranich \cite{MSZ3}, and the oscillation ergodic theorem was obtained by the second author \cite{S}.
    \item The oscillation inequality \eqref{thm:osc} and the jump inequality \eqref{thm:jump} are completely new results for both types of averages when $k''\geq1$ and follows by Theorem~\ref{MAINthm}.
    When $k''=0$, the corresponding results for the jump inequalities are known due to the work of Mirek, Stein, and Zorin-Kranich \cite{MSZ3}. The uniform oscillation inequality was proven by Mirek, Słomian, and Szarek \cite{MSS} in the case of the averages $\calA_t^{\calP,k,0}$ and by the second author \cite{S} in the case of $\calH_t^{\calP,k,0}$.
    
\end{enumerate}
\subsection{Historical background}\label{sec:hist}
In 1931, Birkhoff \cite{Birk} and von Neumann \cite{Neu} proved that the averages
\begin{equation}\label{eq:birkhist}
    M_Nf(x):=\frac{1}{N}\sum_{n=1}^Nf(S^nx)
\end{equation}
converge pointwise $\mu$-almost everywhere on $X$ and in $L^p(X,\mu)$ norm respectively for any $f\in L^p(X,\mu)$, $p \in [1,\infty)$, as $N\to\infty$. In 1955, Cotlar \cite{COT} established the pointwise $\mu$-almost everywhere convergence on $X$ as $N\to\infty$ of the ergodic Hilbert transform given by
\begin{equation*}
    H_Nf(x):=\sum_{1\leq|n|\leq N} \frac{f(S^nx)}{n}
\end{equation*}
for any $f\in L^p(X,\mu)$. In 1968, Calder{\'o}n \cite{Cald} made an important observation (now called the Calder{\'o}n transference principle) that some results in ergodic theory can be easily deduced from known results in harmonic analysis. Namely, the convergence of the Birkhoff averages $M_N$ can be deduced from the boundedness of the Hardy--Littlewood maximal function, and the convergence of Cotlar's averages $H_N$ follows from the boundedness of the maximal function for the truncated discrete Hilbert transform. As we will see ahead, this observation has had a huge impact in the study of convergence problems in ergodic theory.

At the beginning of the 1980's, Bellow \cite{Bel} and independently Furstenberg \cite{Fus} posed the problem about pointwise convergence of the averages along squares given by
\begin{equation*}
    T_N f(x):=\frac{1}{N}\sum_{n=1}^N{f(S^{n^2}x)}.
\end{equation*}
Despite its similarity to Birkhoff's theorem, the problem of pointwise convergence of the $T_N$ averages has a totally different nature from that of its linear counterpart. In particular, the standard approach is insufficient in this case.

We briefly sketch the classical approach of handling the problem of pointwise convergence. It consists of two steps:
\begin{itemize}
    \item[(a)] Establish $L^p$-boundedness for the corresponding maximal function.
    \item[(b)] Find a dense class of functions in $L^p(X,\mu)$ for which the pointwise convergence holds. 
\end{itemize}
In the case of Birkhoff's averages $M_N$, the Calder{\'o}n transference principle allows one to deduce the estimate
\begin{equation*}
\|\sup_{N\in\NN}|M_Nf|\|_{L^p(X,\mu)}\lesssim_{p}\|f\|_{L^p(X,\mu)}
\end{equation*}
for $p\in(1,\infty]$ from the estimate for the discrete Hardy--Littlewood maximal function (and we have a weak-type estimate for $p=1$). In turn, estimates for the discrete Hardy--Littlewood maximal function follow easily from those for the continuous one. This establishes the first step (a). For the second step, one can use the idea of Riesz decomposition \cite{Riesz} to analyze the space ${\II}_S\oplus {\TT}_S\subseteq L^2(X,\mu)$, where
\begin{align*}
\qquad\qquad { \II}_S:=\{f\in L^2(X,\mu): f\circ S =f\}
\qquad \text{ and } \qquad
{\TT}_S:=\{h\circ S-h: h\in L^2(X,\mu)\cap L^{\infty}(X,\mu)\}.
\end{align*}
We see that $M_Nf=f$ for  $f\in {\II}_S$ and, for $g=h\circ S - h\in {\TT}_S$, we have 
\begin{equation*}
M_Ng(x)=\frac{1}{N}\big(h(S^{N+1}x)-h(Sx)\big)
\end{equation*}
by telescoping. Consequently, we see that $M_Ng\to0$ as $N\to\infty$.  This establishes $\mu$-almost everywhere pointwise convergence of $M_N$ on ${\II}_S\oplus {\TT}_S$, which is dense
in $L^2(X,\mu)$. Since $L^2(X,\mu)$ is dense in $L^p(X,\mu)$ for every $p\in[1,\infty)$, this establishes (b).

In the case of the quadratic averages $T_N$, the matter is more complicated. For the first step, by the Calder{\'o}n transference principle, it is enough to establish $\ell^p$ bounds for the maximal function given by
\begin{equation}\label{eq:maximalsq}
    \sup_{N\in\NN}\frac{1}{N}\sum_{n=1}^N f(x-n^2),\quad f\in\ell^p(\ZZ).
\end{equation}
The $\ell^p$ estimate for the above maximal function does not follow directly from the continuous counterpart and requires completely new methods. However, a more serious problem arises in connection with the second step. Namely, the idea of von Neumann fails in this case because the averages $T_N g$ do not possess the telescoping property for $g\in\TT_S$.   

At the end of the 1980's, Bourgain established the pointwise convergence of the averages $T_N$ in a series of groundbreaking articles \cite{B1,B2,B3}. By using the Hardy--Littlewood circle method from analytic number theory, he established $\ell^p$-bounds for the maximal function \eqref{eq:maximalsq}, which establishes step (a). He then bypassed the problem of finding the requisite dense class of functions by using the oscillation seminorm \eqref{def:osc}. Bourgain \cite{B3} proved that,
for any $\lambda>1$ and any sequence of integers $I=(I_j:{j\in\NN})$ with $I_{j+1}>2I_j$ for all $j\in\NN$, we have
\begin{align}\label{eq:non-uni}
\norm[\big]{O_{I, N}^2(T_{\lambda^n}f:n\in\NN)}_{L^2(X,\mu)}
\le C_{I,\lambda}(N)\norm{f}_{L^2(X,\mu)}, \qquad N\in\NN,
\end{align}
for any $f\in L^2(X,\mu)$ with $\lim_{N\to\infty} N^{-1/2}C_{I, \lambda}(N)=0$. This non-uniform inequality \eqref{eq:non-uni} suffices to establish the pointwise convergence of the averaging operators $T_Nf$ for any  $f\in L^2(X, \mu)$. In the same series of papers, by similar methods, Bourgain established the pointwise convergence of the averages along primes
\begin{equation*}
    \frac{1}{|\PP_N|}\sum_{n=1}^Nf(S^nx)\mathds{1}_{\PP}(n)
\end{equation*}
for $f\in L^p(X,\mu)$ with $p>\frac12(1+\sqrt{3})$. In the same year, Wierdl \cite{Wie} extended Bourgain's result to $p\in(1,\infty)$.

In order to establish the inequality \eqref{eq:non-uni}, Bourgain used the Hardy--Littewood circle method and $r$-variation seminorms $V^r$. The $r$-variations were introduced by L{\'e}pingle \cite{Lep} in the context of families of bounded martingales. In 1976, he proved that,
for all $r\in(2, \infty)$,
$p\in(1, \infty)$, and any family of bounded martingales
$(\mathfrak f_n\colon X\to \CC:n\in\NN)$, we have
\begin{align*}
\norm{V^r(\mathfrak f_n: n\in\NN)}_{L^p(X)}\lesssim_{p,r}\sup_{n\in\NN}\norm{\mathfrak f_n}_{L^p(X)}
\end{align*}
with the implicit constant depending only on $p$ and $r$. The above inequality is sharp in the sense that it fails for $r=2$, see \cite{JG} for a counterexample. 

Bourgain observed that the $V^r$ seminorm can be used to obtain \eqref{eq:non-uni}. This is because, by H{\"o}lder's inequality, we have 
\begin{equation*}
    O_{I,N}^2(T_n f:n\in\NN)\le N^{1/2-1/r}V^r(T_n f:n\in\NN)
\end{equation*}
for $r\ge2$. In order to prove the $r$-variational inequality for the averages $T_N$, Bourgain used the $\lambda$-jump counting function. It can easily be seen that 
\begin{align*}
\sup_{\lambda>0}\|\lambda N_{\lambda}(T_N f:N\in\NN)^{1/2}\|_{L^p(X,\mu)}\le \|V^r(T_N f:N\in\NN)\|_{L^p(X,\mu)}
\end{align*}
for every $r\ge2$. The above inequality can be reversed in some sense \cite{B3}. Namely, for any $p\in(1,\infty)$ and any $r\in(2,\infty)$, we have
\begin{align*}
\|V^r(T_N f:N\in\NN)\|_{L^{p,\infty}(X,\mu)}\lesssim_{p,r}
\sup_{\lambda>0}\|\lambda N_{\lambda}(T_N f:N\in\NN)^{1/2}\|_{L^{p,\infty}(X,\mu)}.
\end{align*}
For more details about oscillation, variation, and jump seminorms, we refer to \cite{MSS,MISZWR} and \cite{jsw}. 

The above arguments demonstrate that the problem of proving pointwise convergence can be reduced to proving an appropriate $r$-variational estimate or jump inequality. However, an intriguing question was the issue of uniformity in the inequality \eqref{eq:non-uni}. Shortly after the groundbreaking work of Bourgain, Lacey \cite[Theorem 4.23, p. 95]{RW} improved inequality \eqref{eq:non-uni} showing that, for every $\lambda>1$,
there is a constant $C_{\lambda}>0$ such that
\begin{align}
\label{eq:Lacey}
\sup_{N\in\NN}\sup_{I\in \mathfrak S_N(\LL_{\tau})}\norm[\big]{O_{I, N}^2(T_{\lambda^n}f:n\in\NN)}_{L^2(X)}
\le C_{\lambda}\norm{f}_{L^2(X)},\quad f\in L^2(X,\mu),
\end{align}
where $\LL_{\tau}:=\{\tau^n:n\in \NN\}$. This result motivated the question about uniform estimates independent of $\lambda>1$ in \eqref{eq:Lacey}. In the case of Birkhoff's averages, this question was explicitly formulated in \cite[Problem 4.12, p. 80]{RW}.

In 1998, Jones, Kaufman, Rosenblatt, and
Wierdl \cite{jkrw} established the uniform oscillation inequality on $L^p(X,\mu)$ for the standard Birkhoff averages $M_N$. Two years later, Campbell, Jones, Reinhold, and Wierdl \cite{CJRW1} established the uniform  oscillation inequality for the ergodic Hilbert transform. In 2003, Jones, Rosenblatt, and Wierdl \cite{jrw} proved uniform oscillation inequalities on $L^p(X,\mu)$ with $p\in(1,2]$ for the Birkhoff averages over cubes. However, the case of polynomial averages, even one-dimensional, was open until recent works \cite{MSS,S}, and the case of averages along primes was open until this paper.

In 2015, Mirek and Trojan \cite{MT}, using the ideas of Bourgain and Wierdl, established $\mu$-almost everywhere pointwise convergence of the Cotlar averages along the primes, 
\begin{equation*}
    \sum_{p\in(\pm\PP_N)} \frac{f(S^p)}{p}\log |p|.
\end{equation*}
They proved that the corresponding maximal function is bounded on $L^p(X,\mu)$ with $p>1$ and showed that the analogue of Bourgain's non-uniform oscillation inequality \eqref{eq:non-uni} holds for those averages.

In the same year, Zorin-Kranich \cite{zk} established the pointwise convergence of the averages related to the polynomial mapping given by
\begin{equation*}
    \Tilde{\calP}=(n,n^2,n^3,\ldots,n^d)\colon\ZZ\to\ZZ^d.
\end{equation*}
Namely, he proved that, for any $r>2$ and $|\frac{1}{p}-\frac{1}{2}|<\frac{1}{2(d+1)}$, we have the following $r$-variational estimate
\begin{equation*}
\|V^r(\calA_N^{\Tilde{\calP},1,0}f:N\in\NN)\|_{L^p(X,\mu)}\lesssim_{p,r}\|f\|_{L^p(X,\mu)}.
\end{equation*}
As a consequence, the averages $\calA_N^{\Tilde{\calP},1,0}f$ converge $\mu$-almost everywhere for any $f\in L^p(X,\mu)$. 

In 2016, Mirek and Trojan \cite{MT1} established the pointwise convergence for the averages \eqref{def:average} taken over cubes with $k'=k$, that is
\begin{equation*}
    A_{N,{\rm cube}}^{\calP,k,0} f(x):=\frac{1}{N^k}\sum_{y\in[0,N]^k\cap\ZZ^k}f(S_1^{\calP_1(y)}S_2^{\calP_1(y)}\cdots S_d^{\calP_d(y)}x).
\end{equation*}
There, Mirek and Trojan noted for the first time that the Rademacher--Menshov inequality~\eqref{eq:remark3} may be used to establish $r$-variational estimates. For $p\in(1,\infty)$ and $r>\max\{p, p/(p-1)\}$, they proved that
\begin{equation*}
    \|V^r(A_{N,{\rm cube}}^{\calP,k,0} f:N\in\NN)\|_{L^p(X,\mu)}\leq  C_{p,d,k,{\rm deg}\calP}\|f\|_{L^p(X,\mu)}.
\end{equation*}

Unfortunately, the methods introduced by Bourgain had limitations. These work perfectly fine in the case of the $L^2$ estimates, but, in the case of an $L^p$ estimates with $p\neq2$, there arise difficulties which are hard to overcome concerning the fractions around which major arcs are defined. However, Ionescu and Wainger \cite{IW}, in their groundbreaking 2005 work about discrete singular Radon operators, introduced a set of fractions for which the circle method can be applied towards $L^p$ estimates with $p\neq2$. 

In 2015, Mirek \cite{M} built a discrete counterpart of the Littlewood--Paley theory using the Ionescu--Wainger multiplier theorem and used it to reprove the main result from \cite{IW}. In 2017, Mirek, Stein, and Trojan \cite{MST1,MST2} further exploited these ideas together with the Rademacher--Menshov inequality from \cite{MT1} to obtain an $L^p$ estimate for the $r$-variation seminorm for both $\calA_t^{\calP,k,0}$ and $\calH_t^{\calP,k,0}$ associated with convex sets in the full range of parameters. Namely, they showed that
\begin{equation}\label{eq:var:MST}
\big\|V^r(\calM_t^{\calP,k,0}f: t>0)\big\|_{L^p(X, \mu)}\lesssim_{d,k,p, r,  \deg \mathcal P}\|f\|_{L^p(X, \mu)}
\end{equation}
for $p\in(1,\infty)$ and $r\in(2,\infty)$, where $\calM_t^{\calP,k,0}$ is either $\calA_t^{\calP,k,0}$ or $\calH_t^{\calP,k,0}$. There, the operators $\calH_t^{\calP,k,0}$ are related to Calder{\'o}n--Zygmund kernels satisfying the gradient condition \eqref{eq:gradient}.

In 2019, Trojan \cite{Troj} proved an $L^p$ estimate for the $r$-variation seminorm for both $\calA_t^{\calP,k',k''}$ and $\calH_t^{\calP,k',k''}$ with $k',k''\in\{0,1,\ldots,k\}$ such that $k'+k''=k$. Namely, he showed that
\begin{equation}\label{eq:trojan:var}
\big\|V^r(\calM_t^{\calP,k',k''}f: t>0)\big\|_{L^p(X, \mu)}\lesssim_{d,k,p, r,  \deg \mathcal P}\|f\|_{L^p(X, \mu)}
\end{equation}
for $p\in(1,\infty)$ and $r\in(2,\infty)$, where $\calM_t^{\calP,k',k''}$ is either $\calA_t^{\calP,k',k''}$ or $\calH_t^{\calP,k',k''}$. A straightforward consequence of the inequality \eqref{eq:trojan:var} is the $\mu$-almost everywhere convergence of the averages $\calM_t^{\calP,k',k''}f$. Again, the operators $\calH_t^{\calP,k',k''}$ there are related to Calder{\'o}n--Zygmund kernels satisfying the gradient condition \eqref{eq:gradient}.

In 2020, Mirek, Stein, and Zorin-Kranich \cite{MSZ3} further refined the methods developed in \cite{MST1,MST2} and proved a uniform $L^p$ estimate for the $\lambda$-jump counting function. They proved that  
\begin{equation}\label{eq:jump:MSZK}
    \sup_{\lambda>0}\norm[\big]{\lambda N_{\lambda}(\calM_t^{\calP,k,0}f:t>0)^{1/2}}_{L^p(X,\mu)}\le C_{p,d,k,{\rm deg}\calP}\|f\|_{L^p(X,\mu)}
\end{equation}
for any $p\in(1,\infty)$ and any $f\in L^p(X,\mu)$, where $\calM_t^{\calP,k,0}f$ is either $\calA_t^{\calP,k,0}f$ or $\calH_t^{\calP,k,0}f$. There, the operators $\calH_t^{\calP,k,0}$ are associated with Calder{\'o}n--Zygmund kernels satisfying the H{\"o}lder continuity condition generalizing \eqref{eq:K-modulus-cont}:
For some $\sigma \in (0,1]$ and for every $x, y\in\RR^k\setminus\{0\}$ with $2|y|\leq |x|$, we have
\begin{equation}
\label{eq:K-Holder}
\abs{K(x)-K(x+y)}
\lesssim 
\abs{y}^{\sigma} \abs{x}^{-(k+\sigma)}.
\end{equation}
It is worth noting that the inequality \eqref{eq:jump:MSZK} implies the $r$-variation inequality \eqref{eq:var:MST}. 

In 2021, the second author in collaboration with Mirek and Szarek~\cite{MSS} established the oscillation inequality
\begin{equation}\label{osc:MSS}
\sup_{N\in\NN}\sup_{I\in\mathfrak{S}_N(\RR_+)}\norm[\big]{O_{I,N}^2(\calA_t^{\calP,k,0}f:t>0)}_{L^p(X,\mu)}\le C_{p,d,k,{\rm deg}\calP}\norm{f}_{L^p(X,\mu)},
\end{equation}
and, recently, the second author \cite{S} proved the counterpart of \eqref{osc:MSS} in the case of the operators $\calH_t^{\calP,k,0}$ related to Calder{\'o}n--Zygmund kernels satisfying \eqref{eq:K-Holder}.
\section{Notation and necessary tools}
\subsection{Basic notation}\label{sec:notation}
We denote $\NN:=\{1, 2, \ldots\}$, $\NN_0:=\{0,1,2,\ldots\}$, and $\RR_+:=(0, \infty)$. For
$d\in\NN$, the sets $\ZZ^d$, $\RR^d$, $\CC^d$, and $\TT^d = (\RR/\ZZ)^d \equiv[-1/2, 1/2)^d$ have the standard meanings. For each $N\in\NN$, we set
    \[
    \NN_N:=\{1,\ldots, N\}.
    \]  
For any $x\in\RR$, we set
\[
\lfloor x \rfloor: = \max\{ n \in \ZZ : n \le x \}.
\]
For $u\in\NN$, we define the set
\begin{align*}
     2^{u\NN}:=\{2^{un}\colon n\in\NN\}.
\end{align*}

For two non-negative numbers $A$ and $B$, we write $A \lesssim B$ to indicate that $A\le CB$ for some $C>0$ that may change from line to line, and we may write $\lesssim_{\delta}$ if the implicit constant depends on $\delta$.

We denote the standard inner product on $\RR^d$ by $ x\cdot\xi$. Moreover, for any $x\in\RR^d$, we denote the $\ell^2$-norm and the maximum norm respectively by
\begin{align*}
\abs{x}:=\abs{x}_2:=\sqrt{\ipr{x}{x}} \qquad \text{ and } \qquad |x|_{\infty}:=\max_{1\leq k\leq d}|x_k|.
\end{align*}

For a multi-index $\gamma=(\gamma_1,\dots,\gamma_k)\in\N^k_0$, we abuse the notation to write
$|\gamma|:=\gamma_1+\cdots+\gamma_k$. No confusion should arise since all multi-indices will be denoted by $\gamma$.
\subsection{Seminorms}\label{notation:semi}
Let $\II \subseteq \RR$ and $\lambda > 0$. For $N\in\NN\cup\{\infty\}$, we denote by $\mathfrak S_N(\II)$ the family of all strictly increasing
sequences of length $N+1$ contained in $\II$. We already defined the oscillation seminorm \eqref{def:osc} and the $\lambda$-jump counting function \eqref{def:jump} in the introduction. For any $r\in[1,\infty)$, the $r$-\textit{variation seminorm} $V^r$ of a function $f\colon\II\to\CC$ is defined by
\begin{equation}\label{def:var}
	V^r(f(t) : t \in \II): = \sup_{\substack{t_0 < t_1 < \cdots < t_J \\ t_j \in \II}}
	\Big(\sum_{j = 1}^J \abs{f(t_j) - f(t_{j-1})}^r\Big)^{1/r}.
\end{equation}
The $r$-variational seminorm controls the oscillation seminorm and the $\lambda$-jump counting function. Indeed, by H\"older's inequality, we have
\begin{equation}\label{oscildom}
	O_{I,N}^2(f(t) : t \in \II) \leq N^{1/2-1/r} V^r(f(t) : t \in \II)
\end{equation}
for any $N\in\NN$, $I\in\mathfrak{S}_N(\II)$, and $r \geq 2$. Moreover, for any $\lambda>0$, we have
\begin{equation}\label{jumpdom}
\lambda N_\lambda(f(t) : t \in \II)^{1/r} \leq V^r(f(t) : t \in \II).
\end{equation}
We adopt notation to simultaneously handle the oscillation seminorm and the $\lambda$-jump counting function for the sake of brevity and to emphasize the required properties. Let $E$ be either of $\RR^d$ or $\ZZ^d$ with the usual measures and let $(f_t:t\in\II) \subset L^p(E)$. We write 
\[\calS_E^p(f_t : t \in \II)\]
to represent either of the following quantities:
\[\sup_{N \in \NN}\sup_{I \in \frkS_N(\II)} \big\| O_{I,N}^2(f_t(x) : t \in \II) \big\|_{L^p(E)}\quad\text{or}\quad\sup_{\lambda > 0} \big\| \lambda N_\lambda(f_t(x) : t \in \II)^{1/2} \big\|_{L^p(E)}.\]
\begin{proposition}
\label{prop:subadditive}
Let $p\in(1,\infty)$ and $\II\subseteq\RR$. The seminorm $\calS_E^p$ is subadditive up to a positive constant, that is,
\begin{equation*}
    \calS_E^p(f_t+g_t:t\in\II)\lesssim  \calS_E^p(f_t:t\in\II)+ \calS_E^p(g_t:t\in\II),
\end{equation*}
where the implied constant is independent of $\II$ and the families $(f_t:t\in\II)$ and $(g_t:t\in\II)$. 
\end{proposition}
The critical point is that the jump quasi-seminorm admits an equivalent subadditive seminorm, see \cite[Corollary 2.11]{MSZ3}.

\begin{remark}[Rademacher--Menshov inequality]\label{rem:RM} 
By inequalities \eqref{oscildom} and \eqref{jumpdom}, we deduce that the Rademacher--Menshov inequality \cite[Lemma 2.5, p. 534]{MSZ2} holds for  $\calS_E^p$. Namely, for any $k,m\in\NN$ with $k<2^m$ and any sequence of functions $(f_n:n\in\NN)\subset L^p(E)$, we have
\begin{equation}\label{eq:remark3}
    \calS_E^p(f_n:k\le n\le 2^m)\le\sqrt{2}\norm[\Big]{\sum_{i=1}^s\Big(\sum_{j}{|f_{u_{j+1}^i}-f_{u_j^i}|^2}\Big)^{1/2}}_{L^p(E)},
\end{equation}
where each $[u_j^i,u_{j+1}^i)$ is a dyadic interval contained in $[k, 2^m]$ of the form $[j2^i,(j+1)2^{i})$ for some $0\le i\le m$ and $0\le j\le 2^{m-i}-1$.
\end{remark}
For more information about the $\lambda$-jump counting function and the oscillation and $r$-variation seminorms, we refer to \cite{B3,jsw,MISZWR,MSZ1,S1}.

\subsection{Reductions: Calder{\'o}n transference and lifting}\label{sec:red}
By the Calder\'on transference principle \cite{Cald}, we may restrict attention to the model dynamical system of $\ZZ^d$ equipped with the counting measure and the shift operators $S_j\colon\ZZ^d\to\ZZ^d$ given by $S_j(x_1,\ldots,x_d):=(x_1,\ldots,x_j-1,\ldots,x_d)$. We denote the corresponding averaging operators by
\[ 
A_t^{\calP,k',k''} f(x) = 
	\frac{1}{\vartheta_\Omega(t)}
	\sum_{(n,p)\in\ZZ^{k'}\times(\pm\PP)^{k''}}
	f\big(x - \calP(n, p)\big)\ind{\Omega_t}(n, p) 
	\Big(\prod_{j = 1}^{k''} \log |p_j|\Big) 
\]
and
\[ 
H_t^{\calP,k',k''} f(x) = \sum_{(n,p)\in\ZZ^{k'}\times(\pm\PP)^{k''}}
	f\big(x - \calP(n, p)\big) K(n, p) \ind{\Omega_t}(n, p) \Big(\prod_{j = 1}^{k''} \log \abs{p_j}\Big). 
\]

Moreover, by a standard lifting argument, it suffices to
prove Theorem~\ref{MAINthm} for a canonical case of the polynomial mapping $\calP$. Let $\mathcal{P}$ be a polynomial mapping as in \eqref{polymap}. We define 
\begin{equation*}
    {\rm deg}\, \mathcal{P}:=\max\{{\rm deg}\, \mathcal{P}_j: 1\le j\le d\}
\end{equation*}
and consider the set of multi-indices 
\begin{equation*}
    \Gamma:=\big\{\gamma\in\N_0^k\setminus\{0\}: 0<|\gamma|\le {\rm deg}\, \mathcal{P}\big\}
\end{equation*}
equipped with the lexicographic order. 
We define the \emph{canonical polynomial mapping} by
\begin{equation}\label{canonicalpolymap}
    \RR^k\ni x=(x_1,\dots,x_k)\mapsto\calQ(x):=(x^\gamma\colon\gamma\in\Gamma)\in\RR^\Gamma,
\end{equation} 
where $x^\gamma=x_1^{\gamma_1}x_2^{\gamma_2}\cdots x_k^{\gamma_k}$.
By invoking the lifting procedure described in \cite[Lemma 2.2]{MST1} (see also \cite[Section 11]{bigs}), the following implies Theorem \ref{MAINthm}.
\begin{theorem}\label{maintheorem2}
Let $k\in\NN$, let $\Gamma \subset \NN^{k} \setminus \{0\}$ be a nonempty finite set, and let $k',k''\in\{0,1,\ldots,k\}$ with $k'+k''=k$. Let $M_t^{k',k''}$ be either $A_t^{\calQ,k',k''}$ or $H_t^{\calQ,k',k''}$. For any $p\in(1,\infty)$, there is a constant $C_{p,k,|\Gamma|}>0$ such that 
\begin{equation}
    \calS_{\ZZ^\Gamma}^p(M_t^{k',k''}f:t>0)\leq C_{p,k,|\Gamma|}\|f\|_{\ell^p(\ZZ^\Gamma)}.
\end{equation}
\end{theorem}
\subsection{Fourier transform and Ionescu--Wainger multiplier theorem} 

Let $\GG=\RR^d$ or $\GG=\ZZ^d$ and let $\GG^*$ denote the dual group of $\GG$. For every $z\in\CC$, we set $\ex(z):=e^{2\pi {\bm i} z}$, where ${\bm i}^2=-1$. Let $\calF_{\GG}$ denote the Fourier transform on $\GG$ defined for any $f \in L^1(\GG)$ by
\begin{align*}
\calF_{\GG} f(\xi) := \int_{\GG} f(x) \ex(x\cdot\xi) {\rm d}\mu(x),\quad \xi\in\GG^*,
\end{align*}
where $\mu$ is the usual Haar measure on $\GG$. For any bounded function $\mathfrak m\colon \GG^*\to\CC$, we define the corresponding Fourier multiplier operator by 
\begin{align}
\label{eq:fourier mult}
T_{\GG}[\mathfrak m]f(x):=\int_{\GG^*}\ex(-\xi\cdot x)\mathfrak m(\xi)\calF_{\GG}f(\xi){\rm d}\xi, \quad x\in\GG.
\end{align}
Here, we assume that $f\colon\GG\to\CC$ is a compactly supported function on $\GG$ (and smooth if $\GG=\RR^d$) or any other function for which \eqref{eq:fourier mult} makes sense.

An indispensable tool in the proof of Theorem~\ref{maintheorem2} is the vector-valued Ionescu--Wainger multiplier theorem from \cite[Section 2]{MSZ3} with an improvement by Tao \cite{TaoIW}.
\begin{theorem}\label{IW}
For every $\varrho>0$, there exists a family $(P_{\leq N})_{N\in\NN}$ of subsets of $\NN$ such that:  
\begin{enumerate}[label*={(\roman*)}]
\item \label{IW1} $\NN_N\subseteq P_{\leq N}\subseteq\NN_{\max\{N, e^{N^{\varrho}}\}}$.
\item \label{IW2}  If $N_1\le N_2$, then $P_{\leq N_1}\subseteq P_{\leq N_2}$.
\item \label{IW3}  If $q \in P_{\leq N}$, then all factors of $q$ also lie in $P_{\leq N}$.
\item \label{IW4} $\lcm(P_N) \leq 3^N$. 
\end{enumerate}

Furthermore, for every $p \in (1,\infty)$, there exists  $0<C_{p, \varrho, |\Gamma|}<\infty$ such that, for every $N\in\NN$, the following holds:

Let $0<\varepsilon_N \le e^{-N^{2\varrho}}$ and let $\mathbf Q:=[-1/2, 1/2)^\Gamma$ be a unit cube. Let $\frkM \colon\RR^{\Gamma} \to L(H_0,H_1)$ be a measurable function supported on $\varepsilon_{N}\mathbf Q$ taking values in $L(H_{0},H_{1})$, the space  of bounded linear operators between separable Hilbert spaces $H_{0}$ and $H_{1}$.
Let $0 \leq \mathbf A_{p} \leq \infty$ denote the smallest constant such that
\[
\big\|T_{\RR^\Gamma}[\frkM ]f\big\|_{L^{p}(\RR^{\Gamma};H_1)}
\leq
\mathbf A_{p} \|f\|_{L^{p}(\RR^{\Gamma};H_0)}
\]
for every function $f\in L^2(\RR^\Gamma;H_0)\cap L^{p}(\RR^\Gamma;H_0)$. Then, the multiplier
\[
\Delta_N(\xi)
:=\sum_{b \in\Sigma_{\leq N}}
\frkM (\xi - b),
\]
where $\Sigma_{\leq N}$ is defined by
\[
\Sigma_{\leq N} := \Big\{ \frac{a}{q}\in\QQ^\Gamma\cap\TT^\Gamma:  q \in P_{\leq N}\text{ and } {\rm gcd}(a, q)=1\Big\},
\]
satisfies
\begin{equation}\label{eq:iw:main}
\big\|T_{\ZZ^\Gamma}[\Delta_{N}]f\big\|_{\ell^p(\ZZ^{\Gamma};H_1)}
\le C_{p,\varrho,|\Gamma|} (\log N) \mathbf A_{p}
\|f\|_{\ell^p(\ZZ^{\Gamma};H_0)}
\end{equation}
for every $f\in \ell^p(\ZZ^\Gamma;H_0)$, (cf. \cite[Theorem 1.4]{TaoIW} which removes the factor of $\log N$ in the inequality~\eqref{eq:iw:main}).
\end{theorem}
\section{Preliminaries}
\subsection{General results} In this section, we present some general results concerning the behavior of exponential sums. The following proposition is an enhancement of the variant of Weyl's inequality due to Trojan \cite[Theorem 2]{Troj} that allows us to estimate exponential sums related to a possibly non-differentiable function $\phi$, (cf. \cite[Theorem A.1]{MSZ3}).
\begin{proposition}[Weyl's inequality]\label{weylinq}
Let $\alpha >0$, $k\in\NN$, and let $\Gamma \subset \NN^{k} \setminus \{0\}$ be a nonempty finite set. Let $\Omega'\subseteq\Omega\subseteq B(0,N)\subset\RR^k$ be convex sets and let $\phi\colon \Omega\cap\ZZ^{k}\to\CC$. There is $\beta_\alpha>0$ such that, for any $\beta > \beta_\alpha$, if there is a multi-index $\gamma_0\in\Gamma$ with
\begin{align*}
\abs[\Big]{\xi_{\gamma_0} - \frac{a}{q}}
\leq
\frac{1}{q^2}
\end{align*}
for some coprime integers $a$ and $q$ with $1\leq a\leq q$ and $(\log N)^\beta\leq q\leq N^{|\gamma_0|}(\log N)^{-\beta}$, then
\begin{align*}
\abs[\Big]{\sum_{(n,p)\in\ZZ^{k'}\times(\pm\PP)^{k''}}\ex(\xi\cdot\calQ(n,p))\phi(n,p)\mathds{1}_{\Omega\setminus\Omega'}(n,p)}
&\lesssim
N^k\log(N)^{-\alpha}\norm{\phi}_{L^\infty(\Omega\setminus\Omega')} \\
&\,\,+N^{k} \sup_{\substack{\abs{x-y}\leq N(\log N)^{-\alpha}\\x,y\in\Omega\setminus\Omega'}} \abs{\phi(x)-\phi(y)}.
\end{align*}
The implicit constant is independent of the function $\phi$, the variable $\xi$, the sets $\Omega,\Omega'$, and the numbers $a$, $q$, and $N$.
\end{proposition}
\begin{proof}
We define $\Tilde{\phi}(n,p,A):=\phi(n,p)\mathds{1}_{(\Omega\setminus\Omega')\cap A}(n,p)$.  We partition the cube $[-N,N]^k$ into $J\lesssim\log(N)^{k\alpha}$ cubes $Q_j$ with disjoint interiors and side lengths $C N\log(N)^{-\alpha}$ for some constant $C>0$. Let $(m_j,p_j)$ be a fixed element of $Q_j\cap\Omega\setminus\Omega'$. Since $\mathds{1}_{\Omega}(x)\mathds{1}_{\Omega\setminus\Omega'}(x)=\mathds{1}_{\Omega\setminus\Omega'}(x)$ for any $x\in\RR^k$,  we have
\begin{align}\label{weylone}
    \abs[\Big]{\sum_{(n,p)}\ex(\xi\cdot\calQ(n,p))\Tilde{\phi}(n,p,\Omega)\mathds{1}_{\Omega}(n,p)}\lesssim\sum_{j=1}^J\abs[\Big]{\sum_{(n,p)}\ex(\xi\cdot\calQ(n,p))\Tilde{\phi}(n,p,Q_j)\mathds{1}_{\Omega\cap Q_j}(n,p)},
\end{align}
where all sums are taken over $(n,p)\in\ZZ^{k'}\times(\pm\PP)^{k''}$. Let $(m_j,p_j)$ be a fixed element of $Q_j\cap\Omega\setminus\Omega'$. We estimate the right hand side of \eqref{weylone} by
\begin{align*}
   \sum_{j=1}^J\abs[\Big]{\sum_{(n,p)}\ex(\xi\cdot\calQ(n,p))\Tilde{\phi}(m_j,p_j,Q_j)&\mathds{1}_{\Omega\cap Q_j}(n,p)}\\+
   &\sum_{j=1}^J\abs[\Big]{\sum_{(n,p)}\ex(\xi\cdot\calQ(n,p))\big(\Tilde{\phi}(m_j,p_j,Q_j)-\Tilde{\phi}(n,p,Q_j)\big)\mathds{1}_{\Omega\cap Q_j}(n,p)}.
\end{align*}
By Trojan's variant of Weyl's inequality \cite[Theorem 2]{Troj},
the first term is bounded by
\begin{equation}\label{eq:weyl:new}
    J N^k(\log N)^{-\alpha'}\norm{\phi}_{L^\infty(\Omega\setminus\Omega')}\lesssim N^k(\log N)^{k\alpha-\alpha'}\norm{\phi}_{L^\infty(\Omega\setminus\Omega')}
\end{equation}
for any $\alpha'>0$. Since $\mathds{1}_{\Omega\cap Q_j}(n,p)=\mathds{1}_{\Omega'\cap Q_j}(n,p)+\mathds{1}_{(\Omega\setminus\Omega')\cap Q_j}(n,p)$, the second term is bounded by
\begin{equation}\label{eq:weyl:new2}
    N^k(\log N)^{k\alpha-\alpha'}\norm{\phi}_{L^\infty(\Omega\setminus\Omega')}+N^k\sup_{\substack{\abs{x-y}\leq N(\log N)^{-\alpha}\\x,y\in\Omega\setminus\Omega'}} \abs{\phi(x)-\phi(y)}.
\end{equation}
Choosing an appropriate $\alpha'>0$ in \eqref{eq:weyl:new} and \eqref{eq:weyl:new2} yields the claim.
\end{proof}
The next result is a generalization of \cite[Proposition 4.1]{Troj} and \cite[Proposition 4.2]{Troj} in the spirit of \cite[Proposition 4.18]{MSZ3}. For $q\in\NN$ and $a\in\NN_q^\Gamma$ with ${\rm gcd}(a,q)=1$, the \emph{Gaussian sum} related to the polynomial mapping $\calQ$ is given by
\begin{equation}
G(a/q) := \frac{1}{q^{k'}} \frac{1}{\varphi(q)^{k''}} \sum_{x \in \NN^{k'}_q} \sum_{y \in A_q^{k''}} 
	\ex((a/q)\cdot\calQ(x, y)),
\end{equation}
where $A_q:=\{a\in\NN_q:{\rm gdc}(a,q)=1\}$ and $\varphi$ is Euler's totient function. There is $\delta > 0$ such that
\begin{equation}\label{gausssumest}
    \big|G(a/q) \big| \lesssim q^{-\delta},
\end{equation}
according to \cite[Theorem 3]{Troj}.

\begin{lemma}\label{approxlemma}
Let $N\in\NN$ and let $\Omega\subseteq B(0,N)\subset\RR^k$ be a convex set or a Boolean combination
of finitely many convex sets. Let $\calK\colon\RR^k\to\CC$ be a continuous function supported in $\Omega$. Then, for each $\beta>0$, there is a constant $c = c_{\beta}>0$ such that, for any $q \in \NN$ with $1 \leq q \leq (\log N)^{\beta}$, $a \in A_q$, and $\xi = a/q + \theta \in \RR^\Gamma$, we have
\begin{align*}
\bigg|\sum_{(n,p)\in\ZZ^{k'}\times(\pm\PP)^{k''}}\ex\big(\xi\cdot\calQ(n,p)\big)\calK(n,p)\Big(\prod_{i=1}^{k''}\log|p_i|\Big)-G(a/q)\int_{\Omega}\ex\big((\xi-a/q)\cdot\calQ(t)\big)\calK(t){\rm d }t\bigg|
\\ \lesssim \big[N^{k-1}\|\calK\|_{L^{\infty}(\Omega)} \big(1 
		+
		\sum_{\gamma \in \Gamma} |\theta_\gamma| N^{|\gamma|} \big)
 + N^k\sup_{\substack{x,y \in \Omega \\ |x-y| \leq q\sqrt{k}}} |\calK(x) - \calK(y)|\big]N\exp\big(-c\sqrt{\log N}\big). 
\end{align*}
The implied constant is independent of $N,a,q,\xi$ and the kernel $\calK$.
\end{lemma}
\begin{proof}
The case when $k=k'$ was proven in \cite[Proposition 4.18]{MSZ3}, so we assume that $k>k'$. Observe that, for a prime number $p$, $p \mid q$ if and only if $(p \bmod q, q) > 1$. Hence, for each 
	$s \in \{1, \ldots, k''\}$, we have
	\begin{equation}
        \begin{aligned}
        \label{eq:coprime}
		&\bigg|
		\sum_{n \in \NN_0^{k'}}
		\sum_{\substack{r'' \in \NN_q^{k''} \\ (r_s'', q) > 1}}
		\sum_{\substack{p \in \PP^{k''} \\ p \equiv r'' \bmod q}}
		\ex(\xi \cdot \calQ(n, p)) \calK(n, p)
		\bigg(\prod_{j=1}^{k''} \log p_j\bigg)
		\bigg|
		\\ &\lesssim
		N^{k-1} \|\calK\|_{L^{\infty}(\Omega)}
		\sum_{p \mid q} \log p
		\lesssim
		N^{k-1} \|\calK\|_{L^{\infty}(\Omega)}\log q
            \lesssim N^{k-1} \|\calK\|_{L^{\infty}(\Omega)}\log \log(N).
        \end{aligned}
	\end{equation} 
To simplify the notation, for $(x,y) \in \RR^k \setminus \{0\}$, we set
 $F(x,y):= \ex(\theta \cdot \calQ(x, y))\calK(x, y).$ For $(n, p) \in \NN^{k'} \times \PP^{k''}$ with $n \equiv r' \bmod q$ and 
	$p \equiv r'' \bmod q$, we have
	\[
		\begin{aligned}
		\xi_\gamma n^{\gamma'} p^{\gamma''} \equiv \frac{a_\gamma}{q} n^{\gamma'} p^{\gamma''} 
		+ \theta_\gamma n^{\gamma'} p^{\gamma''} 
		\equiv \frac{a_\gamma}{q} (r')^{\gamma'} (r'')^{\gamma''} + \theta_\gamma n^{\gamma'} p^{\gamma''} \pmod 1.
		\end{aligned}
	\]
	Therefore, we have $\ex(\xi \cdot \calQ(n, p)) = \ex((a/q) \cdot \calQ(r', r''))\ex(\theta \cdot \calQ(n, p)),$ so then
	\begin{equation}
        \begin{aligned}
        \label{eq:conjclass}
		\sum_{n \in \NN_0^{k'}} \sum_{p \in \PP^{k''}} \ex(\xi \cdot \calQ(n, p)) &\calK(n, p)
		\bigg(\prod_{j = 1}^{k''} \log p_j\bigg)
		\\
		&=
		\sum_{r' \in \NN^{k'}_q} 
		\sum_{r'' \in A_q^{k''}}
		\ex((a/q) \cdot \calQ(r', r'')) 
		\sum_{\substack{n \in \NN_0^{k'} \\ n \equiv r' \bmod q}}
		\sum_{\substack{p \in \PP^{k''} \\ p \equiv r'' \bmod q}}
		F(n,p)
		\bigg(\prod_{j = 1}^{k''} \log p_j\bigg)
		\\ & +
		\calO\Big(N^{k-1}\|\calK\|_{L^{\infty}(\Omega)} \log \log N\Big),
            \end{aligned}
		\end{equation}
 where the error term is the cost for making the summation for $r''$ over $A_q^{k''}$ instead of $N_q^{k''}$. 
	Fix $u \in \NN^{k'}$, $\tilde{p} \in \PP^{k''-1}$, and $r_1'' \in A_q$. Then
		$\left\{v \in \NN : (u, v, \tilde{p}) \in \Omega \right\} = \left(V_0+1, \ldots, V_1\right)$
	for some $0 \leq V_0 \leq V_1 \leq N$. 
	By partial summation, we obtain
	\begin{equation}
		\label{sumparts1}
		\begin{aligned}
		&\sum_{\substack{p_1 \in \PP_{V_1} \setminus \PP_{V_0} \\ p_1 \equiv r_1'' \bmod q}}
        F(u,p_1, \tilde{p}) \log p_1 
		=
		\sum_{\substack{v_1 = V_0+1 \\ v_1 \equiv r_1'' \bmod q}}^{V_1}
		F(u,v_1, \tilde{p}) 
		\ind{\PP}(v_1) \log v_1 \\
		&=
		\vartheta(V_1; q, r''_1) F(u,V_1,\tilde{p})
		-
		\vartheta(V_0; q, r''_1) F(u,V_0+1,\tilde{p}) 
		- \sum_{v_1 = V_0+1}^{V_1-1} \vartheta(v_1; q, r_1'') 
		[F(u,v_1+1,\tilde{p}) - F(u,v_1,\tilde{p})] ,
		\end{aligned}
	\end{equation}
	where, for $x \geq 1$, we have set
	\[
		\vartheta(x; q, r) := \sum_{\substack{p \in \PP_x \\ p \equiv r \bmod q}} \log p.
	\]
	Similarly, we have
	\begin{equation}
		\label{sumparts2}
		\sum_{v_1 = V_0 + 1}^{V_1}
        F(u,v_1, \tilde{p})
		=
		V_1 F(u,V_1,\tilde{p})
		-
		V_0 F(u,V_0+1,\tilde{p}) 
		- \sum_{v_1 = V_0+1}^{V_1-1} v_1 
		[F(u,v_1+1,\tilde{p}) - F(u,v_1,\tilde{p})].
	\end{equation}
	Furthermore, in view of the Siegel--Walfisz theorem (\cite{Sieg, Wal}, see also \cite[Corollary 11.21]{MV}), 
	there are $C, c' > 0$ such that for all $x \geq 1$, $(r, q) = 1$ and $1 \leq q \leq (\log x)^{\beta'}$,
	\begin{equation}
		\label{Sieg-Wal}
		\bigg|\vartheta(x; q, r) - \frac{x}{\varphi(q)} \bigg| \leq C x \exp\big(-c' \sqrt{\log x}\big).
	\end{equation}
	Hence, by \eqref{sumparts1}, \eqref{sumparts2}, and \eqref{Sieg-Wal}, we obtain
	\begin{align*}
		&\bigg| \sum_{\substack{p_1 \in \PP_{V_1}\setminus\PP_{V_0} \\ p_1 \equiv r_1'' \bmod q}}
        F(u, p_1, \tilde{p}) \log p_1 -
		\frac{1}{\varphi(q)}
		\sum_{v_1 = V_0 +1}^{V_1} F(u, v_1, \tilde{p})
		\bigg| 
		 \\ & \qquad \lesssim
		\|\calK\|_{L^{\infty} (\Omega)} \bigg|\vartheta(V_1; q, r''_1) - \frac{V_1}{\varphi(q)}\bigg|
		+ 
		\|\calK\|_{L^{\infty} (\Omega)} \bigg|\vartheta(V_0; q, r''_1) - \frac{V_0}{\varphi(q)}\bigg| 
		\\ & \qquad \phantom{\lesssim} +
		\bigg[\|\calK\|_{L^\infty(\Omega)}\sum_{\gamma \in \Gamma} |\theta_\gamma| N^{|\gamma|-1} 
 + \sup_{\substack{x,y \in \Omega \\ |x-y| \leq 1}} |\calK(x) - \calK(y)|\bigg] 
		\sum_{v_1 = V_0+1}^{V_1-1}
		\bigg|\vartheta(v_1; q, r''_1) - \frac{v_1}{\varphi(q)} \bigg|  \\
		& \qquad \lesssim
		\big[\|\calK\|_{L^{\infty}(\Omega)} \big(1 
		+
		\sum_{\gamma \in \Gamma} |\theta_\gamma| N^{|\gamma|} \big)
 + N\sup_{\substack{x,y \in \Omega \\ |x-y| \leq 1}} |\calK(x) - \calK(y)|\big] 
		N\exp\big(-c'\sqrt{\log N}\big).
	\end{align*}
Similar arguments applied to the sums over $p_2, \ldots, p_{k''}$ give 
	\begin{equation}
        \begin{aligned}
        \label{eq:recursion}
		\bigg| \sum_{\substack{u \in \NN_0^{k'} \\ u \equiv r' \bmod q}}
		&\sum_{\substack{p \in \PP^{k''} \\ p \equiv r'' \bmod q}}
		F(u,p)
		\bigg(\prod_{j = 1}^{k''} \log p_j \bigg) -
		\frac{1}{\varphi(q)^{k''}}
		\sum_{u \in \NN_0^{k'}}
		\sum_{v \in \NN^{k''}}
		F(qu+r',v) \bigg|
		\\ & \lesssim
		\big[N^{k-1}\|\calK\|_{L^{\infty}(\Omega)} \big(1 
		+
		\sum_{\gamma \in \Gamma} |\theta_\gamma| N^{|\gamma|} \big)
 + N^k\sup_{\substack{x,y \in \Omega \\ |x-y| \leq 1}} |\calK(x) - \calK(y)|\big]N\exp\big(-c'\sqrt{\log N}\big). 
        \end{aligned}
	\end{equation}
Let $\Omega_+:= \Omega \cap [0,\infty)^{k'} \times [1,\infty)^{k''}$. We can estimate the sum by an integral by writing
\begin{equation}
\begin{aligned}
\label{eq:sum int}
&\bigg| q^{k'}\sum_{u \in \NN_0^{k'}}\sum_{v \in \NN^{k''}} F(qu+r',v) - \iint_{\Omega_+} F(s,t) \textrm{d}s \textrm{d}t \bigg| 
\\ &= \bigg| q^{k'}\sum_{u \in \NN_0^{k'}}\sum_{v \in \NN^{k''}} F(qu+r',v) - \sum_{u \in \NN_0^{k'}}\sum_{v \in \NN^{k''}} \int_{qu+[0,q)^{k'}}\int_{v+[0,1)^{k''}} F(s,t) \textrm{d}s \textrm{d}t \bigg| 
\\ &\leq \sum_{u \in \NN_0^{k'}}\sum_{v \in \NN^{k''}} \int_{[0,q)^{k'}}\int_{[0,1)^{k''}} |F(qu+r',v) - F(qu+s,v+t)| \textrm{d}s \textrm{d}t.
\end{aligned}
\end{equation}
We use three estimates to control this:
\begin{align*}
    |\ex(\theta \cdot \calQ(qu+r',v)) - \ex(\theta \cdot \calQ(qu+s,v+t))| &\lesssim \sum_{\gamma \in \Gamma} q|\theta_\gamma|N^{|\gamma|-1},\\
    |\calK(qu+r',v)-\calK(qu+s,v+t)| &\lesssim \sup_{\substack{x,y\in\Omega \\ |x-y|\leq q\sqrt{k}}} |\calK(x) - \calK(y)|,\\
    \sum_{u \in \NN_0^{k'}} \sum_{v \in \NN^{k''}} |\ind{\Omega}(qu+r',v) - \ind{\Omega}(qu+s,v+t)| &\lesssim (N/q)^{k-1},
\end{align*}
where the last inequality is a consequence of \cite[Proposition 4.16]{MSZ3}, which gives that the number of lattice points in $\Omega$ at a distance $<q$ from the boundary of $\Omega$ is $\calO(q N^{k-1})$. We therefore get a bound for \eqref{eq:sum int} of the form
\begin{align*}
\calO\bigg(qN^{k-1}\|\calK\|_{L^{\infty}(\Omega)}\bigg[1 +  \sum_{\gamma \in \Gamma} |\theta_\gamma|N^{|\gamma|} \bigg] + N^k\sup_{\substack{x,y\in\Omega \\ |x-y|\leq q\sqrt{k}}} |\calK(x) - \calK(y)|\bigg).
\end{align*}
Applying this in \eqref{eq:recursion} and combining the error terms appropriately gives
\begin{align*}
\bigg| \sum_{\substack{u \in \NN_0^{k'} \\ u \equiv r' \bmod q}}
		&\sum_{\substack{p \in \PP^{k''} \\ p \equiv r'' \bmod q}}
		F(u,p)
		\bigg(\prod_{j = 1}^{k''} \log p_j \bigg) -
		\frac{1}{q^{k'}}
            \frac{1}{\varphi(q)^{k''}}
		\iint_{\Omega_+} F(s,t) \textrm{d}s \textrm{d}t  \bigg|
		\\ & \lesssim
		\big[N^{k-1}\|\calK\|_{L^{\infty}(\Omega)} \big(1 
		+
		\sum_{\gamma \in \Gamma} |\theta_\gamma| N^{|\gamma|} \big)
 + N^k\sup_{\substack{x,y \in \Omega \\ |x-y| \leq q\sqrt{k}}} |\calK(x) - \calK(y)|\big]N\exp\big(-c'\sqrt{\log N}\big).
\end{align*}
Applying this in \eqref{eq:conjclass} by summing in $r'$ and $r''$ together with \eqref{eq:coprime} gives
\begin{equation}\label{approx:final}
\begin{aligned}
\bigg|\sum_{(n,p)\in\NN_0^{k'}\times\PP^{k''}}\ex\big(\xi\cdot\calQ(n,p)\big)\calK(n,p)\mathds{1}_{\Omega}(n,p)\Big(\prod_{i=1}^{k''}\log|p_i|\Big)-G(a/q)\int_{\Omega_+}\ex\big((\xi-a/q)\cdot\calQ(t)\big)\calK(t){\rm d }t\bigg|
\\ \lesssim \big[N^{k-1}\|\calK\|_{L^{\infty}(\Omega)} \big(1 
		+
		\sum_{\gamma \in \Gamma} |\theta_\gamma| N^{|\gamma|} \big)
 + N^k\sup_{\substack{x,y \in \Omega \\ |x-y| \leq q\sqrt{k}}} |\calK(x) - \calK(y)|\big]N\exp\big(-c\sqrt{\log N}\big)
\end{aligned}
\end{equation}
for any $c < c'$. In simplifying to get the error term above, note that 
$q^{k'} \phi(q)^{k''} \leq q^{k} \leq (\log N)^{\beta k}$
and
\[(\log N)^{\beta k}\exp(-c'\sqrt{\log N}) \lesssim \exp(-c\sqrt{\log N}).\]
Finally, we note that we can increase the range of integration at \eqref{approx:final} to the larger $\Omega \cap [0,\infty)^{k}$ by noting that
\begin{equation*}
G(a/q)\int_{\Omega\cap[0,\infty]^{k'}\times[0,1)^{k''}}\ex\big((\xi-a/q)\cdot\calQ(t)\big)\calK(t){\rm d}t
\end{equation*}
is bounded by $N^{k'}\|\calK\|_{L^\infty(\Omega)}\leq N^{k-1}\|\calK\|_{L^\infty(\Omega)}$.

We can repeat the entire proof replacing $\NN_0$ with $-\NN_0$ and/or $\PP$ with $-\PP$ in all the $2^k$ many possible combinations thereof in $\NN_0^{k'} \times \PP^{k''}$. Then, collecting all of the error terms yields the claim.    
\end{proof}

\subsection{Multipliers for the averaging operators}
For a function $f\colon\ZZ^\Gamma \rightarrow \CC$ with finite support, we have
\[A_t^{\calQ,k',k''}f(x) = T_{\ZZ^\Gamma}[\frkM_{t}]f(x) \quad\text{and}\quad H_t^{\calQ,k',k''}f(x) = T_{\ZZ^\Gamma}[\mathfrak{n}_{t}]f(x)\]
for the discrete Fourier multipliers
\begin{equation*}
\frakm_{t}(\xi):=
\frac{1}{\vartheta_\Omega(t)}\sum_{(n,p)\in\ZZ^{k'}\times(\pm\PP)^{k''}}\ex\big(\xi\cdot\calQ(n,p)\big)\mathds{1}_{\Omega_t}(n,p)\Big(\prod_{i=1}^{k''}\log|p_i|\Big),\quad \xi\in\TT^\Gamma,
\end{equation*}
and
\begin{equation*}
\mathfrak{n}_t(\xi):=
\sum_{(n,p)\in\ZZ^{k'}\times(\pm\PP)^{k''}}\ex(\xi\cdot\calQ(n,p))K(n,p)\mathds{1}_{\Omega_t}(n,p)\Big(\prod_{i=1}^{k''}\log|p_i|\Big),\quad \xi\in\TT^\Gamma.
\end{equation*}
Their continuous counterparts are given by
\begin{equation*}
    \Phi_t(\xi):=\frac{1}{|\Omega_t|}\int_{\Omega_t}\ex(\xi\cdot\calQ(t)){\rm d} t\quad\text{and}\quad \Psi_t(\xi):={\rm p.v.}\int_{\Omega_t}\ex(\xi\cdot\calQ(t))K(t){\rm d} t
\end{equation*}
respectively. To present a unified approach, we write $M_t^{k',k''}$, $\frkY_t$, and $\Theta_t$ to represent either $A_t^{\calQ,k',k''}$, $\frkM_t$, and $\Phi_t$ or $H_t^{\calQ,k',k''}$, $\frkN_t$, and $\Psi_t$ respectively. 
 We now present the key properties of our multiplier operators that will be used in the proof of Theorem~\ref{maintheorem2}. Let $N_n:=\lfloor 2^{n^\tau} \rfloor$ for $n\in\NN$ and some $\tau\in(0,1]$ adjusted later.

\begin{enumerate}[label={\bf Property \arabic*.}, ref=\arabic*, itemindent=*,leftmargin=0pt]
\item\label{pr:1} For each $\alpha > 0$, there is $\beta_{\alpha} > 0$ such that, for any
	$\beta > \beta_{\alpha}$ and $n \in \NN$, if there is a multi-index $\gamma_0 \in \Gamma$ with
	\[
    	\bigg|\xi_{\gamma_0} - \frac{a}{q} \bigg| \leq \frac{1}{q^2}
	\]
	for some coprime integers $a$ and $q$ with $1 \leq a \leq q$ and
	$(\log N_n)^\beta \leq q \leq N_n^{\abs{\gamma_0}} (\log N_n)^{-\beta}$, then
	\[
		|(\frkY_{N_n} - \frkY_{N_{n-1}})(\xi)| \lesssim C(\log N_n)^{-\alpha}.
	\]
 This follows from Proposition~\ref{weylinq} with $\phi(x)\equiv(\vartheta_\Omega(N_n))^{-1}$ for the $\frkY_t = \frkM_{t}$ case and with $\phi(x)=K(x)$ for the $\frkY_t = \frkN_{t}$ case, noting the size condition \eqref{eq:size-unif} and the continuity condition \eqref{eq:K-modulus-cont}.
\item\label{pr:2}
Let $A$ be the $|\Gamma| \times |\Gamma|$ diagonal matrix with
\begin{equation}
	\label{matrixA}
	(A v)_\gamma = \abs{\gamma} v_\gamma.
\end{equation}
For any $t>0$, we set $t^A v: = \big(t^{\abs{\gamma}} v_\gamma : \gamma \in \Gamma\big).$ Then
\begin{equation*}
\big|\Theta_{N_n}(\xi) - \Theta_{N_{n-1}}(\xi)\big|
		\lesssim
		\min\big\{|N_n^A\xi|_\infty, |N_n^A \xi|_\infty^{-1/|\Gamma|}\big\},\quad\text{for each }n\in\NN.
\end{equation*}
In the $\Theta_t = \Phi_t$ case, this follows from the mean value theorem and the standard van der Corput lemma. In the $\Theta_t = \Psi_t$ case, this follows from the cancellation condition \eqref{eq:cancel} and \cite[Proposition B.2]{MSZ2} (see \cite[p. 21]{MSZ2} for details).

\item\label{pr:3}
For each $\alpha > 0$,  $n \in \NN$, and $\xi \in \TT^\Gamma$ satisfying
\[
    \bigg|\xi_\gamma - \frac{a_\gamma}{q} \bigg| \leq N_n^{-\abs{\gamma}} L\qquad\text{for all }\gamma \in \Gamma
\]
with $1 \leq q \leq L$, $a\in A_q^\Gamma$, and $1 \leq L \leq \exp\big(c\sqrt{\log{N_n}}\big) (\log N_n)^{-\alpha}$, we have
\[
        \frkY_{N_n}(\xi) - \frkY_{N_{n-1}}(\xi) =
        G(a/q) 
		\big(\Theta_{N_n}(\xi - a/q) - \Theta_{N_{n-1}}(\xi - a/q)\big) 
		+ \mathcal{O}\big((\log N_n)^{-\alpha}\big),
\]
for some constant $c>0$ which is independent of $n, \xi, a$ and $q$. 

In the $\frkY_t = \frkM_t$, $\Theta_t = \Phi_t$ case, this is \cite[Property 6]{Troj}. In the $\frkY_t = \frkN_t$, $\Theta_t = \Psi_t$ case, this follows from Property~\ref{pr:1} alongside Lemma~\ref{approxlemma} with $\Omega:=\Omega_{N_n}\setminus\Omega_{N_{n-1}}$ and $\calK(n,p):=K(n,p)\ind{\Omega}$, noting the size condition \eqref{eq:size-unif} and the continuity condition \eqref{eq:K-modulus-cont}. For details see \cite[Lemmas 3 and 5]{Troj}.
\end{enumerate}

\subsection{Parameters discussion}
Let $p\in(1,\infty)$ be fixed and let $\chi\in(0,1/10)$. Fix $\tau$ with $0 < \tau < 1-\min(2,p)^{-1}$ and let $N_n:=\lfloor 2^{n^\tau} \rfloor$ for $n\in\NN$. If $p \in (1,2)$, fix $p_0$ such that $1 < p_0 < p$. If instead $p \in (2,\infty)$, fix $p_0 > p$. If $p=2$, the discussion is moot since all the interpolation arguments in the article become unnecessary. We choose $\rho$ with 
\[\rho > \frac{1}{\tau}\frac{pp_0-2p}{2p_0 - 2p}\]
so that interpolation of the estimates
\[\|T\|_{\ell^2} \lesssim n^{-\rho \tau} \quad \text{and} \quad \|T\|_{\ell^{p_0}}\lesssim 1 \]
yields 
\[\|T\|_{\ell^p} \lesssim n^{-(1+\varepsilon)} \text{ for some } \varepsilon > 0.\]
Property~\ref{pr:1} gives us a corresponding $\beta_\rho$. We fix a choice of $\beta > \beta_\rho$ and then fix a choice of $u\in\NN$ with $u>|\Gamma|\beta$. We also have the value of $\delta$ coming from the Gaussian sum estimate \eqref{gausssumest}. With these fixed, we choose the value of $\varrho$ in Theorem~\ref{IW} to be
\[\varrho:= \min\bigg(\frac{\chi}{10 u}, \frac{\delta}{8\tau}\bigg).\]
\section{Proof of Theorem~\ref{maintheorem2}}
By the monotone convergence theorem and standard density arguments it is enough to prove that
\[
\calS_{\ZZ^\Gamma}^p(M_t^{k',k''}f : t\in\II) \lesssim_{p,k,|\Gamma|} \|f\|_{\ell^p(\ZZ^\Gamma)}
\]
holds for every finite subset $\II\subset\RR_+$ with the implicit constant independent of the set $\II$. We start by splitting (cf. \cite[Lemma 1.3]{jsw}) into long oscillations/jumps and short variations along the subexponential sequence $N_n$: 
\[
\calS_{\ZZ^\Gamma}^p\big(M_t^{k',k''}f : t\in\II\big)\lesssim \calS_{\ZZ^\Gamma}^p(T_{\ZZ^\Gamma}[\frkY_{N_n}]f : n \in \NN_0) + 
\bigg\| \Big(\sum_{n\in\NN_0} V^2\big(M_t^{k',k''}f : t \in [N_{n}, N_{n+1})\cap\II\big)^2 \Big)^{1/2} \bigg\|_{\ell^p(\ZZ^\Gamma)}.
\]
\subsection{Short variations}
By using the arguments from \cite[Section 3.1]{MSZ3}, the estimate for the short variations will follow from the estimate
\begin{equation}\label{eq:short1}
\big\|V^1(M_t^{k',k''}f:t\in [N_{n}, N_{n+1})\cap\II)\big\|_{\ell^1(\ZZ^\Gamma)}\lesssim n^{\tau-1}\|f\|_{\ell^1(\ZZ^\Gamma)}.
\end{equation}
Let $t_1<t_2<\cdots<t_{J(n)}$ be a sequence of elements of $[N_n,N_{n+1})\cap\II$. Since the number of elements in $[N_n,N_{n+1})\cap\II$ is finite, it is easy to see that
\begin{equation*}
\big\|V^1(M_t^{k',k''}f:t\in [N_{n}, N_{n+1})\cap\II)\big\|_{\ell^1(\ZZ^\Gamma)}\leq\Big\|\sum_{j=1}^{J(n)}\big|M_{t_j}^{k',k''}f-M_{t_{j-1}}^{k',k''}f\big|\Big\|_{\ell^1(\ZZ^\Gamma)}
\end{equation*}
for any $n\in\NN_0.$ Moreover, we have 
\begin{equation}\label{eq:short2}
\Big\|\sum_{j=1}^{J(n)}\big|M_{t_j}^{k',k''}f-M_{t_{j-1}}^{k',k''}f\big|\Big\|_{\ell^1(\ZZ^\Gamma)}\lesssim 2^{-kn^\tau}\big(\vartheta_\Omega(N_{n+1})-\vartheta_\Omega(N_n)\big)\|f\|_{\ell^1(\ZZ^\Gamma)}.
\end{equation}
This follows from the monotonicity of the sets $\Omega_t$
and having $\vartheta_\Omega(t)\approx t^k$ by the prime number theorem in the $M_t^{k',k''} = A_t^{\calQ, k',k''}$ case or the size condition \eqref{eq:size-unif} in the $M_t^{k',k''} = H_t^{\calQ, k',k''}$ case. By \cite[Eq. 4.10]{Troj}, the right hand side of \eqref{eq:short2} is bounded by $n^{\tau-1}\|f\|_{\ell^1(\ZZ^\Gamma)}$, proving \eqref{eq:short1}.
\subsection{Long oscillations/jumps and the circle method}
Let $\eta\colon\RR^\Gamma \rightarrow [0,1]$  be a smooth function with
\[
	\eta(x) = 
	\begin{cases}
		1 & \text{if } |x|_\infty \leq \tfrac{1}{32 |\Gamma|}, \\
		0 & \text{if } |x|_\infty \geq \tfrac{1}{16 |\Gamma|}. 
	\end{cases}
\]
For $N\in\RR_+$, we define the scaling notation
\[
\eta_N(\xi):= \eta\big(2^{N \cdot A- N^{\chi}\cdot \rm{Id}}\xi\big)
\]
where $A$ is the matrix given in \eqref{matrixA} and $\rm{Id}$ is the $|\Gamma| \times |\Gamma|$ identity matrix. For dyadic integers $s \in 2^{u\NN}$, we define the \textit{annuli sets of fractions} by
\begin{equation}\label{annsets}
\Sigma_s := \begin{cases} \Sigma_{\leq s} & \text{ if } s=2^u, \\ \Sigma_{\leq s} \setminus \Sigma_{\leq s/2^u} & \text{ if } s>2^u, \end{cases}
\end{equation}
where the $\Sigma_{\leq \cdot}$ are the sets of Ionescu--Wainger fractions as in Theorem~\ref{IW}.
For $t\geq 2^u$, we set $F(t):=\max\{s \in 2^{u\NN} : s \leq t\}$.
We define
\[
\Xi_{\leq j^{\tau u}}(\xi) :=\sum_{a/q \in \Sigma_{\leq F(j^{\tau u})}} \eta_{j^{\tau}}(\xi - a/q)
\] 
and, for $s \in 2^{u\NN}$, we define the \textit{annuli functions}
\begin{equation}\label{annmulti}
\Xi_j^s(\xi):= \sum_{a/q \in \Sigma_{s}} \eta_{j^{\tau}}(\xi - a/q).
\end{equation}
By \eqref{annsets}, we have the telescoping property
\[
\Xi_{\leq j^{\tau u}}=\sum_{\substack{s \in 2^{u\NN} \\ s \leq j^{\tau u}}} \Xi_j^s.
\]
Note that $\eta_{j^\tau}(\xi)$ satisfies the hypothesis about the support for $\frkM$ in Theorem~\ref{IW} since $\frac{1}{8 |\Gamma|}2^{-j^{\tau}+j^{\tau\chi}} \leq e^{-j^{2 \tau u \varrho}}$ provided that $\varrho \leq \chi/(10 u)$. Using the $\Xi_{\leq j^{\tau u}}$ functions, we bound the long oscillations/jumps by
\begin{align*}
\calS_{\ZZ^\Gamma}^p\Big(\sum_{j=1}^n T_{\ZZ^\Gamma}[(\frkY_{N_j} - \frkY_{N_{j-1}})\Xi_{\leq j^{\tau u}}]f : n \in \NN\Big) + \calS_{\ZZ^\Gamma}^p\Big(\sum_{j=1}^n T_{\ZZ^\Gamma}[(\frkY_{N_j} - \frkY_{N_{j-1}})(1-\Xi_{\leq j^{\tau u}})]f : n \in \NN\Big). 
\end{align*}
These terms correspond to major and minor arcs respectively. 

\subsection{Minor arcs}
Since $V_1$ controls the oscillation/jump seminorms, we have 
\begin{align*}
\calS_{\ZZ^\Gamma}^p\Big(\sum_{j=1}^n T_{\ZZ^\Gamma}[(\frkY_{N_j} - \frkY_{N_{j-1}})(1-\Xi_{\leq j^{\tau u}})]f : n \in \NN\Big) 
\leq \sum_{n=1}^\infty \big\| T_{\ZZ^\Gamma}[(\frkY_{N_n} - \frkY_{N_{n-1}})(1-\Xi_{\leq n^{\tau u}})]f \big\|_{\ell^p(\ZZ^\Gamma)}.
\end{align*}
It then suffices to show that
\begin{equation*}
    \big\| T_{\ZZ^\Gamma}[(\frkY_{N_n} - \frkY_{N_{n-1}})(1-\Xi_{\leq n^{\tau u}})]f \big\|_{\ell^{p}(\ZZ^\Gamma)} \lesssim n^{-(1+\varepsilon)}\|f\|_{\ell^p(\ZZ^\Gamma)}
\end{equation*}
for some $\varepsilon > 0$. This uses Property~\ref{pr:1} and follows from the proof of \cite[Eqs. (5.8), (5.9)]{Troj} with only small changes due to our differing scaling in the definition of $\eta_N(\xi)$. We omit the details.  
\subsection{Introduction to major arcs}
Using the annuli multipliers \eqref{annmulti} and Proposition \eqref{prop:subadditive}, we bound the major arcs term by
\begin{align*}
\calS_{\ZZ^\Gamma}^p\Big(\sum_{j=1}^n \sum_{\substack{s \in 2^{u\NN} \\ s\leq j^{\tau u}}} T_{\ZZ^\Gamma}[(\frkY_{N_j} - \frkY_{N_{j-1}})\Xi_j^s]f : n \in \NN\Big)
 \leq \sum_{s \in 2^{u \NN}} \calS_{\ZZ^\Gamma}^p\Big(\sum_{\substack{1 \leq j \leq n \\ j \geq s^{1/(\tau u)}}}  T_{\ZZ^\Gamma}[(\frkY_{N_j} - \frkY_{N_{j-1}})\Xi_j^s]f : n \geq s^{1/\tau u} \Big).
\end{align*}
It then suffices to show for large $s\in 2^{u\NN}$ that
\begin{equation}\label{eq:major1}
\calS_{\ZZ^\Gamma}^p\bigg(\sum_{\substack{1 \leq j \leq n \\ j \geq s^{1/(\tau u)}}} T_{\ZZ^\Gamma}[(\frkY_{N_j} - \frkY_{N_{j-1}})\Xi_j^s]f : n \geq s^{1/\tau u}\bigg) \lesssim s^{-\varepsilon} \|f\|_{\ell^p(\ZZ^\Gamma)}
\end{equation}
for some $\varepsilon > 0$ since $\sum_{s \in 2^{u\NN}} s^{-\varepsilon} < \infty$. Let $\kappa_s:=s^{2 \lfloor \varrho \rfloor}$. By splitting the left hand side of \eqref{eq:major1} at $n \approx 2^{\kappa_s}$ into small and large scales, it suffices to prove that
\begin{equation}\label{small}
\calS_{\ZZ^\Gamma}^p\Big(\sum_{\substack{1 \leq j \leq n \\ j \geq s^{1/(\tau u)}}} T_{\ZZ^\Gamma}[(\frkY_{N_j} - \frkY_{N_{j-1}})\Xi_j^s]f : n^\tau \in [s^{1/u}, 2^{\kappa_s+1}] \Big) \lesssim s^{-\varepsilon}\|f\|_{\ell^p(\ZZ^\Gamma)}
\end{equation}
and
\begin{equation}
\label{large}
\calS_{\ZZ^\Gamma}^p\Big(\sum_{\substack{1 \leq j \leq n \\ j \geq 2^{\kappa_s/\tau}}}  T_{\ZZ^\Gamma}[(\frkY_{N_j} - \frkY_{N_{j-1}})\Xi_j^s]f : n^\tau > 2^{\kappa_s} \Big) \lesssim s^{-\varepsilon} \|f\|_{\ell^p(\ZZ^\Gamma)}.
\end{equation}

For the small scales \eqref{small}, we will use the Rademacher--Menshov inequality \eqref{eq:remark3} and Theorem~\ref{IW}. For the large scales \eqref{large}, we will use the Magyar--Stein--Wainger sampling principle from \cite[Proposition 2.1]{MSW} and its counterpart for the jump inequality from \cite[Theorem 1.7]{MSZ1}. We first establish an approximation lemma to replace our discrete multipliers with continuous counterparts. Let
\begin{equation}
    v_j^s(\xi):=\sum_{a/q \in \Sigma_s} G(a/q)\big(\Theta_{N_j} - \Theta_{N_{j-1}}\big)(\xi - a/q)\eta_{j^\tau}(\xi-a/q)
\end{equation}
and
\begin{equation}
    \Lambda_j^s(\xi):= \sum_{a/q \in \Sigma_s} \big(\Theta_{N_j} - \Theta_{N_{j-1}}\big)(\xi - a/q)\eta_{j^\tau}(\xi-a/q).
\end{equation}

\begin{lemma}
\label{multiplierapprox}
Let $M \in \NN$, $\alpha' > 0$, and $S_M:=\lfloor 2^{M^\tau - 3M^{\tau \chi}}\rfloor$. For $j\in\NN$ with $s^{1/(\tau u)} \leq j$ and $M \leq j \leq 2M$, we have
\begin{equation}
\label{approx1}
\|(\frkY_{N_j}-\frkY_{N_{j-1}})\Xi_j^s - v_j^s\|_{\ell^\infty(\TT^\Gamma)} \lesssim j^{-\alpha' \tau}
\end{equation}
and
\begin{equation}
\label{approx2}
\|(\frkY_{N_j}-\frkY_{N_{j-1}})\Xi_j^s - \Lambda_j^s \frkM_{S_M}\|_{\ell^\infty(\TT^\Gamma)} \lesssim j^{-\alpha' \tau}.
\end{equation}
\end{lemma}

\begin{proof}
For \eqref{approx1}, since the $\eta_{j^{\tau}}(\xi - a/q)$ bump functions in the definitions of $\Xi_j^s$ and $v_j^s$ have disjoint supports for distinct fractions $a/q$, it suffices to prove for a fixed $a/q \in \Sigma_s$ that
\[\big\| \eta_{j^{\tau}}(\xi - a/q)\big[\big(\frkY_{N_j}-\frkY_{N_{j-1}}\big)(\xi) -  G(a/q)\big(\Theta_{N_j} - \Theta_{N_{j-1}}\big)(\xi - a/q)\big]\big\|_{\ell^\infty(\TT^\Gamma)} \lesssim j^{-\alpha' \tau }\]
with the implied constant independent of the choice of $a/q$. Using the definition of $\Sigma_s$, property~\textit{\ref{IW1}} from Theorem~\ref{IW}, $s \leq j^{\tau u}$, and $\varrho \leq \chi/(10 u)$, we have 
$q \leq e^{s^{\varrho}} \leq e^{j^{\tau u \rho}}\leq 2^{j^{\tau \chi}} =: L_1.$
On the support of $\eta_{j^\tau}(\xi - a/q)$, we have $|\xi_\gamma - a_\gamma/q| \lesssim N_j^{-|\gamma|} L_1$ for all $\gamma \in \Gamma$. Moreover, we have
\[L_1 = 2^{j^{\tau \chi}} \leq 2^{j^{\tau /2}}j^{-\tau \alpha'}\lesssim \exp(\sqrt{\log N_j})(\log N_j)^{-\alpha'}.\]
The estimate \eqref{approx1} then follows from Property~\ref{pr:3} with $\alpha = \alpha'$ and $L = L_1$.

For~\eqref{approx2}, we use~\eqref{approx1} and are reduced to showing that
\[
\|v_j^s - \Lambda_j^s \frkM_{S_M}\|_{\ell^\infty(\TT^\Gamma)} \lesssim j^{-\alpha' \tau}.
\]
Fixing $\xi$ in the support of $\eta_{j^\tau}(\xi - a/q)$, we have
\begin{align*}
|\xi_\gamma - a_\gamma/q|\leq 2^{-M^\tau |\gamma|} 2^{(2M)^{\tau \chi}} \leq 2^{-M^\tau |\gamma|} 2^{2M^{\tau \chi}}= 2^{-M^\tau |\gamma|} 2^{3M^{\tau \chi}|\gamma|} 2^{-3M^{\tau \chi}|\gamma|} 2^{2M^{\tau \chi}}\leq S_M^{-|\gamma|} 2^{-(j/2)^{\tau \chi}}
\end{align*}
for all $\gamma \in \Gamma$. By the triangle inequality, we have
\[|G(a/q) - \frkM_{S_M}(\xi)| \leq |G(a/q) - G(a/q) \Phi_{S_M}(\xi - a/q)| + |G(a/q) \Phi_{S_M}(\xi - a/q) - \frkM_{S_M}(\xi)|.\]
For the first term, we use the estimate \eqref{gausssumest}, the mean value theorem, and Property \ref{pr:2} to obtain
\begin{align*}
    |G(a/q) - G(a/q) \Phi_{S_M}(\xi - a/q)| \leq q^{-\delta} \big|S_M^A (\xi - a/q)\big|_{\infty}
    \leq q^{-\delta} 2^{-(j/2)^{\tau \chi}} \lesssim j^{-\alpha' \tau}. 
\end{align*}
For the second term, we use that 
\begin{align*}
    |\xi_\gamma - a_\gamma/q| \leq S_M^{-|\gamma|} 2^{-M^{\tau \chi}} \leq S_M^{-|\gamma|} 2^{(2M)^{\tau \chi}} =: S_M^{-|\gamma|} L_2 
\end{align*}
and
$q \leq 2^{j^{\tau \chi}} \leq 2^{(2M)^{\tau \chi}} = L_2 \leq \exp\big(\sqrt{ \log S_M}\big) (\log S_M)^{-\alpha' / \chi}.$
Hence, we may apply \cite[Lemma 3]{Troj} with $\alpha = \alpha'/\chi$, $N = S_M$, and $L = L_2$ to obtain
\begin{align*}
    |G(a/q) \Phi_{S_M}(\xi - a/q) - \frkM_{S_M}(\xi)| \lesssim \log(S_M)^{-\alpha'/\chi} \lesssim (M^\tau - 3M^{\tau \chi})^{-\alpha'/\chi}  \lesssim (2M)^{-\alpha' \tau} \lesssim j^{-\alpha' \tau}.
\end{align*}
This completes the proof of \eqref{approx2}.
\end{proof}
\subsection{Small scales}
 Using that $V_2$ dominates oscillations/jumps, splitting $[s^{1/u},2^{\kappa_s+1}]$ into dyadic intervals, and preparing via the triangle inequality to use \eqref{approx2}, we bound the left hand side of \eqref{small} by
\begin{align*}
\overbrace{\sum_{M \in 2^{\NN} \cap [s^{1/u},2^{\kappa_s}]} \Big\|V_2\Big(\sum_{\substack{1 \leq j \leq n \\ j \geq s^{1/(\tau u)}}} T_{\ZZ^\Gamma}[\Lambda_j^s \frkM_{S_M}]f : n^\tau \in [M,2M] \Big)\Big\|_{\ell^p(\ZZ^\Gamma)}}^{\text{Main Term 1}}
\\ +\underbrace{\sum_{M \in 2^{\NN} \cap [s^{1/u},2^{\kappa_s}]} \Big\|V_2\Big(\sum_{\substack{1 \leq j \leq n \\ j \geq s^{1/(\tau u)}}} T_{\ZZ^\Gamma}[(\frkY_{N_j} - \frkY_{N_{j-1}})\Xi_j^s - \Lambda_j^s \frkM_{S_M}]f : n^\tau \in [M,2M] \Big)\Big\|_{\ell^p(\ZZ^\Gamma)}}_{\text{Error Term 1}}.
\end{align*}
For Error Term~1, it will suffice to show that 
\begin{equation}\label{error1}
    \big\|T_{\ZZ^\Gamma}[(\frkY_{N_n} - \frkY_{N_{n-1}})\Xi_n^s - \Lambda_n^s \frkM_{S_M}]f\big\|_{\ell^p(\ZZ^\Gamma)} \lesssim n^{-(1+\varepsilon')} \|f\|_{\ell^p(\ZZ^\Gamma)}
\end{equation}
for some $\varepsilon' > 0$ since we would then bound it by
\begin{align*}
\sum_{n \geq s^{1/(\tau u)}} n^{-(1+\varepsilon')} \|f\|_{\ell^p(\ZZ^\Gamma)}
\lesssim s^{-\varepsilon'/(\tau u)} \|f\|_{\ell^p(\ZZ^\Gamma)} \lesssim s^{-\varepsilon}\|f\|_{\ell^p(\ZZ^\Gamma)}  
\end{align*}
using that $V_1$ dominates $V_2$. We note by Theorem~\ref{IW} that
\[ \big\|T_{\ZZ^\Gamma}[(\frkY_{N_n} - \frkY_{N_{n-1}})\Xi_n^s - \Lambda_n^s \frkM_{S_M}]f\big\|_{\ell^{p_0}(\ZZ^\Gamma)} \lesssim \|f\|_{\ell^{p_0}(\ZZ^\Gamma)}\]
and, by \eqref{approx2} with $\alpha' = \rho$, that
\[ \big\|T_{\ZZ^\Gamma}[(\frkY_{N_n} - \frkY_{N_{n-1}})\Xi_n^s - \Lambda_n^s \frkM_{S_M}]f\big\|_{\ell^2(\ZZ^\Gamma)} \lesssim n^{-\rho \tau} \|f\|_{\ell^2(\ZZ^\Gamma)}.\]
Interpolation of the above inequalities yields \eqref{error1}.
 
For Main Term~1, we apply the Rademacher--Menshov inequality \eqref{eq:remark3} to bound it by
\begin{align*}
 \sum_{M \in 2^{\NN} \cap [s^{1/u},2^{\kappa_s}]} \sum_{i = 0}^{\log_2(2M)} \bigg\|\Big(\sum_j \Big| \sum_{k \in I_{i,j}^M} T_{\ZZ^\Gamma}[\Lambda_k^s \frkM_{S_M}]f \Big|^2\Big)^{1/2}\bigg\|_{\ell^p(\ZZ^\Gamma)},
\end{align*}
where $j$ is taken over $j \geq 0$ such that $I_{i,j}^M:= [j2^i,(j+1)2^i] \cap [M^{1/\tau}, (2M)^{1/\tau}] \neq \emptyset$. Let $\tilde{\eta}_{N}(\xi):=\eta_N(\xi/2)$. Then $\tilde{\eta}_{N} \eta_{k^\tau} = \eta_{k^\tau}$ for $k^\tau \geq N$ due to the nesting supports. This lets us write
\[\Lambda_k^s \frkM_{S_M} = \Lambda_k^s \frkM_{S_M} \sum_{a/q \in \Sigma_s} \tilde{\eta}_M(\xi - a/q)=:\Lambda_k^s \frkM_{S_M}\tilde{\Xi}_{M^{1/\tau}}^s\]
for $k \in I_{i,j}^M$ since then $k \geq M^{1/\tau}$. We have for any $p\in(1,\infty)$ that
\[\bigg\|\Big(\sum_j \Big| \sum_{k \in I_{i,j}^M} T_{\ZZ^\Gamma}[\Lambda_k^s]g \Big|^2\Big)^{1/2}\bigg\|_{\ell^{p}(\ZZ^\Gamma)} \lesssim \|g\|_{\ell^{p}(\ZZ^\Gamma)}\]
since, by Theorem~\ref{IW}, the above estimate is a consequence of its continuous counterpart
\begin{align*}
\norm[\bigg]{ \Big(\sum_{j}\big|\sum_{k\in I_{i,j}^M}T_{\RR^\Gamma}\big[(\Theta_{N_k}-\Theta_{N_{k-1}})\eta_k^\tau\big]f\big|^2\Big)^{1/2}}_{L^{p}(\RR^\Gamma)}
\lesssim \norm{f}_{L^{p}(\RR^\Gamma)}.
\end{align*}
The above square function estimate follows by  appealing to Property~2 and arguments from Littlewood--Paley theory. We refer to \cite{MSZ2} for more details, see also \cite[Theorem 4.3, p. 42]{MSZ3}. Thus,
\begin{equation}\label{square_estimate_1}
    \bigg\|\Big(\sum_j \Big| \sum_{k \in I_{i,j}^M} T_{\ZZ^\Gamma}[\Lambda_k^s \frkM_{S_M}]f \Big|^2\Big)^{1/2}\bigg\|_{\ell^{p_0}(\ZZ^\Gamma)} 
\lesssim \big\|T_{\ZZ^\Gamma}[\frkM_{S_M}]f\big\|_{\ell^{p_0}(\ZZ^\Gamma)} \lesssim \|f\|_{\ell^{p_0}(\ZZ^\Gamma)}
\end{equation}
using the uniform $\ell^{p}$-boundedness of the averaging operators. We get an improved bound on $\ell^2$. To do this, we show that
\[\big\|\frkM_{S_M} \tilde{\Xi}_{M^{1/\tau}}^s\big\|_{\ell^{\infty}(\TT^\Gamma)}\lesssim s^{-\delta}\]
for $M \in 2^\NN \cap [s^{1/u},2^{\kappa_s}]$. Since the bump functions in the sum have disjoint supports, it suffices to prove for a fixed $a/q \in \Sigma_s$ that 
\[ \|\frkM_{S_M}(\xi) \tilde{\eta}_M(\xi - a/q) \|_{\ell^\infty(\TT^\Gamma)} \lesssim s^{-\delta}\]
with the implied constant independent of the choice of $a/q$. On the support of $\tilde{\eta}_M(\xi - a/q)$, we have
\[|\xi_\gamma - a_\gamma/q| \leq \frac{1}{8|\Gamma|} 2^{-M|\gamma|} 2^{M^{\chi}} \leq 2^{-M^{\tau}|\gamma|} 2^{2M^{\tau \chi}}.\]
We follow the same arguments as in the proof of \eqref{approx2}, choosing $\alpha' = \delta u /\tau$, to show that 
\[ |G(a/q) - \frkM_{S_M}(\xi)| \lesssim M^{-\delta u} \leq s^{-\delta}.\]
For any $\xi \in \TT^\Gamma$ and $a/q \in \Sigma_s$, we have
\[\big|\frkM_{S_M}(\xi) \tilde{\Xi}_{M^{1/\tau}}^s(\xi) \big| \leq |\frkM_{S_M}(\xi)| \leq |\frkM_{S_M}(\xi) - G(a/q)| + |G(a/q)| \lesssim s^{-\delta}\]
using that $|G(a/q)| \lesssim q^{-\delta} \lesssim s^{-\delta}$ since $q \geq s/2^u$ by the construction of $\Sigma_s$. Hence,
\begin{equation}
\label{square_estimate_2}
    \bigg\|\Big(\sum_j \Big| \sum_{k \in I_{i,j}^M} T_{\ZZ^\Gamma}[\Lambda_k^s \frkM_{S_M}]f \Big|^2\Big)^{1/2}\bigg\|_{\ell^2(\ZZ^\Gamma)}
    \lesssim \big\|T_{\ZZ^\Gamma}[\frkM_{S_M}\tilde{\Xi}_{M_{1/\tau}^s}]f\big\|_{\ell^{2}(\ZZ^\Gamma)} \lesssim s^{-\delta}\|f\|_{\ell^{2}(\ZZ^\Gamma)}.
\end{equation}
Interpolation of \eqref{square_estimate_1} with \eqref{square_estimate_2} then gives that 
\[ 
\bigg\|\Big(\sum_j \Big| \sum_{k \in I_{i,j}^M} T_{\ZZ^\Gamma}[\Lambda_k^s \frkM_{S_M}]f \Big|^2\Big)^{1/2}\bigg\|_{\ell^p(\ZZ^\Gamma)}\lesssim s^{-8\varrho}\|f\|_{\ell^p(\ZZ^\Gamma)}
\]
since $8\varrho \leq \delta/(\rho \tau)$. Thus, we may dominate Main Term~1 by
\begin{align*}
\sum_{M \in 2^{\NN} \cap [s^{1/u},2^{\kappa_s}]} \sum_{i = 0}^{\log_2(2M)} s^{-8\varrho}\|f\|_{\ell^p(\ZZ^\Gamma)}\lesssim \kappa_s^2 s^{-8\varrho} \|f\|_{\ell^p(\ZZ^\Gamma)} \lesssim s^{-4\varrho} \|f\|_{\ell^p(\ZZ^\Gamma)}
\end{align*}
since $\kappa_s \leq s^{2\varrho}$, concluding the proof of \eqref{small}.

\subsection{Large scales} 
Since $V_1$ dominates $\calS_{\ZZ^\Gamma}^p$, we may bound the left hand side of \eqref{large} by
\begin{align*}
\overbrace{\calS_{\ZZ^\Gamma}^p\Big(\sum_{\substack{1 \leq j \leq n \\ j \geq 2^{\kappa_s/\tau}}}  T_{\ZZ^\Gamma}[v_j^s]f : n^\tau > 2^{\kappa_s} \Big)}^{\text{Main Term 2}}
+ \overbrace{\sum_{n \geq 2^{\kappa_s/\tau}}  \big\|T_{\ZZ^\Gamma}[(\frkY_{N_n} - \frkY_{N_{n-1}})\Xi_n^s - v_n^s]f \big\|_{\ell^p(\ZZ^\Gamma)}}^{\text{Error Term 2}}.
\end{align*}
For Error Term~2, it will suffice to show that
\begin{equation}\label{error2}
    \big\|T_{\ZZ^\Gamma}[(\frkY_{N_n} - \frkY_{N_{n-1}})\Xi_n^s - v_n^s]f \big\|_{\ell^p(\ZZ^\Gamma)} \lesssim e^{(|\Gamma|+1)s^\varrho} n ^{-(1+\varepsilon')} \|f\|_{\ell^p(\ZZ^\Gamma)}
\end{equation}
for some $\varepsilon' > 0$ since we would then bound it by
\begin{align*}
e^{(|\Gamma|+1)s^\varrho} \sum_{n \geq 2^{\kappa_s/\tau}} n^{-(1+\varepsilon')} \|f\|_{\ell^p(\ZZ^\Gamma)}
\lesssim e^{(|\Gamma|+1)s^\varrho} 2^{-s^{2\varrho}\varepsilon'/\tau} \|f\|_{\ell^p(\ZZ^\Gamma)} \lesssim s^{-\varepsilon} \|f\|_{\ell^p(\ZZ^\Gamma)}.
\end{align*}
We have \[\big\|T_{\ZZ^\Gamma}[(\frkY_{N_n} - \frkY_{N_{n-1}})\Xi_n^s - v_n^s]f \big\|_{\ell^2(\ZZ^\Gamma)} \lesssim n^{-\rho \tau} \|f\|_{\ell^2(\ZZ^\Gamma)}\] 
by \eqref{approx1} with $\alpha' = \rho$. We also have 
\[\big\|T_{\ZZ^\Gamma}[(\frkY_{N_n} - \frkY_{N_{n-1}})\Xi_n^s - v_n^s]f \big\|_{\ell^{p_0}(\ZZ^\Gamma)} \lesssim e^{(|\Gamma|+1)s^\varrho} \|f\|_{\ell^{p_0}(\ZZ^\Gamma)}\] 
simply by the triangle inequality and property \textit{(i)} from Theorem~\ref{IW}. Consequently \eqref{error2} follows by interpolation.

For Main Term~2, we define
\[
w^s(\xi):= \sum_{a/q \in \Sigma_s} G(a/q)\tilde{\eta}_{2^{\kappa_s}}(\xi - a/q), \quad \Pi^s(\xi):=\sum_{a/q \in \Sigma_s} \tilde{\eta}_{2^{\kappa_s}}(\xi - a/q),\]
and
\[
\omega_n^s(\xi):= \sum_{2^{\kappa_s/\tau} \leq j \leq n} (\Theta_{N_j} - \Theta_{N_{j-1}})(\xi)\eta_{j^\tau}(\xi). 
\]

Let $Q_s:=\lcm(q: a/q \in \Sigma_s).$ By property \textit{(iv)} from Theorem~\ref{IW}, we have $Q_s \leq 3^s$. The function $\omega_n^s$ is supported on $[-\frac{1}{4Q_s}, \frac{1}{4Q_s}]$ for large $s\in2^{u\NN}$ since, on the support of $\eta_{2^{\kappa_s}}$, we have $|\xi_\gamma| \leq2^{-2^{-\kappa_s}+2^{\kappa_s\chi}} \leq (4Q_s)^{-1}$
for all $\gamma \in \Gamma$ and large $s$. We also have 
\[\sum_{2^{\kappa_s/\tau} \leq j \leq n} v_j^s(\xi) = w^s(\xi)\sum_{b \in \ZZ^\Gamma} \omega_n^s(\xi - b/Q_s)  .\]
Therefore, it suffices to prove
\begin{equation}\label{largescale1}
\calS_{\ZZ^\Gamma}^p\Big(T_{\ZZ^\Gamma}\Big[\sum_{b \in \ZZ^\Gamma} \omega_n^s(\cdot - b/Q_s) \Big]f : n^\tau > 2^{\kappa_s} \Big)\lesssim\|f\|_{\ell^p(\ZZ^\Gamma)}
\end{equation}
and
\begin{equation}
\label{largescale2}
\big\| T_{\ZZ^\Gamma}[w^s]f\|_{\ell^p(\ZZ^\Gamma)} \lesssim s^{-\varepsilon} \|f\|_{\ell^p(\ZZ^\Gamma)}
\end{equation}
for some $\varepsilon > 0$. 

By the Magyar--Stein--Wainger sampling principle \cite[Proposition 2.1]{MSW} for the oscillation seminorm or the sampling principle for the jumps \cite[Theorem 1.7]{MSZ1}, \eqref{largescale1} follows from
\begin{equation}\label{largescale1bis}
\calS_{\RR^\Gamma}^p(T_{\RR^\Gamma}[\omega_n^s]f : n^\tau > 2^{\kappa_s}) \lesssim \|f\|_{L^p(\RR^\Gamma)}.
\end{equation}
To prove \eqref{largescale1bis}, we use that the $\omega_n^s$ functions are almost telescoping. We define
\[\Delta_n^s(\xi) := \sum_{2^{\kappa_s/\tau} \leq j \leq n} (\Theta_{N_j} - \Theta_{N_{j-1}})(\xi) = (\Theta_{N_n} - \Theta_{N_{2^{\kappa_s/\tau} - 1}})(\xi). \]
Then \eqref{largescale1bis} follows from 
\begin{equation}\label{largecontdelta}
\calS_{\RR^\Gamma}^p(T_{\RR^\Gamma}[\Delta_n^s]f : n^\tau > 2^{\kappa_s})\lesssim\|f\|_{L^p(\RR^\Gamma)}
\end{equation}
since the error term is bounded by
\begin{equation*}
\sum_{n > 2^{\kappa_s/\tau}} \big\|T_{\RR^\Gamma}[(\Theta_{N_n} - \Theta_{N_{n-1}})(\eta_{n^\tau}-1)]f\big\|_{L^p(\RR^\Gamma)}\lesssim\|f\|_{L^p(\RR^\Gamma)}
\end{equation*}
using Property \ref{pr:2} and interpolation.

On the other hand, due to translation invariance of $\calS_{\RR^\Gamma}^p$, ~\eqref{largecontdelta} follows from
\begin{align}\label{largeconttheta}
\calS_{\RR^\Gamma}^p(T_{\RR^\Gamma}[\Theta_t]f : t > 0) \lesssim \|f\|_{L^p(\RR^\Gamma)}.
\end{align}
For the jump inequality, the estimate \eqref{largeconttheta} was proven in \cite[Theorem 1.22, Theorem 1.30]{MSZ2} for both $\Phi_t$ and $\Psi_t$. For the oscillation inequality, \eqref{largeconttheta} was proven in \cite[Eq. 3.38]{MSS} for $\Phi_t$ and in \cite[Theorem 1.9]{S} for $\Psi_t$. This concludes the proof of \eqref{largescale1}. 

For \eqref{largescale2},
we note by \eqref{gausssumest} that
\begin{equation}\label{largescale3}
\big\| T_{\ZZ^\Gamma}[w^s]f\|_{\ell^2(\ZZ^\Gamma)} \lesssim s^{-\delta} \|f\|_{\ell^2(\ZZ^\Gamma)}. 
\end{equation}
On $\ell^{p_0}$, we start by splitting
\[w^s = \Pi^s\frkM_{J_s} + (w^s - \Pi^s\frkM_{J_s}),\]
where $J_s = \lfloor 2^{2^{\kappa_s} - 3 \cdot 2^{\kappa_s \chi}} \rfloor$.
By Theorem~\ref{IW}, we have
\begin{equation}
\label{largescale4}
\big\| T_{\ZZ^\Gamma}[\Pi^s\frkM_{J_s}]f\|_{\ell^{p_0}(\ZZ^\Gamma)} \lesssim \|f\|_{\ell^{p_0}(\ZZ^\Gamma)}.
\end{equation}
Let $p_{00} \in (1,\infty)$. Then 
\begin{equation} 
\label{largescale5}
\big\| T_{\ZZ^\Gamma}[w^s - \Pi^s\frkM_{J_s}]f\|_{\ell^{p_{00}}(\ZZ^\Gamma)} \lesssim e^{(|\Gamma|+1)s^\varrho} \|f\|_{\ell^{p_{00}}(\ZZ^\Gamma)}
\end{equation}
by property \textit{(i)} from Theorem~\ref{IW}. Therefore, it suffices to show that 
\begin{equation}
\label{largescale6}
\big\| T_{\ZZ^\Gamma}[w^s - \Pi^s\frkM_{J_s}]f\|_{\ell^2(\ZZ^\Gamma)} \lesssim 2^{-\chi s^{2\rho}} \|f\|_{\ell^2(\ZZ^\Gamma)}
\end{equation}
since interpolating \eqref{largescale6} with \eqref{largescale5} for an appropriate choice of $p_{00}$ gives 
\[ \big\| T_{\ZZ^\Gamma}[w^s - \Pi^s\frkM_{J_s}]f\|_{\ell^{p_0}(\ZZ^\Gamma)} \lesssim \|f\|_{\ell^{p_0}(\ZZ^\Gamma)}, \]
combining this with \eqref{largescale4} gives
\[ \big\| T_{\ZZ^\Gamma}[w^s]f\|_{\ell^{p_0}(\ZZ^\Gamma)} \lesssim \|f\|_{\ell^{p_0}(\ZZ^\Gamma)}, \]
and interpolating the above inequality with \eqref{largescale3} completes the proof of \eqref{largescale2} and, thereby, that of \eqref{large}. The proof of \eqref{largescale6} will proceed similarly as in the proof of \eqref{approx2}. For $\xi$ in the support of $\tilde{\eta}_{2^{\kappa_s}}(\xi - a/q)$, we have
\begin{align*}
|\xi_\gamma - a_\gamma/q| \leq \frac{1}{16|\Gamma|} 2^{-2^{\kappa_s} |\gamma|} 2^{2^{\kappa_s \chi}} 
\leq 2^{-2^{\kappa_s} |\gamma|} 2^{2^{\kappa_s \chi}}\leq J_s^{-|\gamma|} 2^{-2^{\kappa_s \chi}} 
\end{align*} 
for all $\gamma \in \Gamma$. By the triangle inequality, we have
\[|G(a/q) - \frkM_{J_s}(\xi)| \leq |G(a/q) - G(a/q) \Phi_{J_s}(\xi - a/q)| + |G(a/q) \Phi_{J_s}(\xi - a/q) - \frkM_{J_s}(\xi)|.\]
For the first term, we use ~\eqref{gausssumest} and the mean value theorem to write
\begin{align*}
    |G(a/q) - G(a/q) \Phi_{J_s}(\xi - a/q)| \leq q^{-\delta} \big|J_s^A (\xi - a/q)\big|_{\infty}
    \leq q^{-\delta} 2^{-2^{\kappa_s \chi}} \lesssim 2^{-\chi s^{2\varrho}}. 
\end{align*}
For the second term, we use that
\begin{align*}
    |\xi_\gamma - a_\gamma/q| \leq J_s^{-|\gamma|} 2^{-2^{\kappa_s \chi}} \leq J_s^{-|\gamma|} e^{s^{\varrho}} =: J_s^{-|\gamma|} L_3 
\end{align*}
and
$q \leq L_3 \leq \exp\big(\sqrt{ \log J_s}\big) (\log J_s)^{-1}$
to apply \cite[Lemma 3]{Troj} with $\alpha = 1$, $N = J_s$, and $L = L_3$. This gives
\begin{align*}
    |G(a/q) \Phi_{J_s}(\xi - a/q) - \frkM_{J_s}(\xi)| \lesssim \log(J_s)^{-1} \lesssim (2^{\kappa_s} - 3 \cdot 2^{\kappa_s \chi})^{-1}
     \lesssim 2^{-\chi \kappa_s} \lesssim 2^{-\chi s^{2\varrho}},
\end{align*}
completing the proof. 
\section{Remarks}
As a simple consequence of our results, we can prove the convergence of the Wiener--Wintner type averages. This result is probably known, but we have not found anything like this in the literature in the
presented generality.

Let $(X,\mu)$ be a measure space endowed with a measure preserving transformation $T\colon X\to X$ and let  $$R(x)=a_m x^m+a_{m-1}x^{m-1}+\cdots +a_1x+a_0,\quad x\in\RR,$$
be a polynomial with \textit{real coefficients}. Moreover, let $P\colon\ZZ\to\ZZ$ be a polynomial with \textit{integer coefficients} such that $P(0)=0$. For $p\in(1,\infty)$, the \textit{Wiener--Wintner type averages}
\begin{equation}\label{wienerwinter}
    \frac{1}{2N+1}\sum_{n=-N}^N f(T^{P(n)}x)\ex\big(R(n)\big)
\end{equation}
converge $\mu$-almost everywhere for any $f\in L^p(X,\mu)$. According to Assani \cite[p. 179]{ASS}, the convergence of the averages \eqref{wienerwinter} in the case when $\deg P\geq2$ is known only for $R\equiv0$. However, in \cite[Theorem 1.9]{EK}, the authors have established the convergence in the case when $P(n)=n^2$ and $R(n)=\theta n$ for any $\theta\in\RR.$

Let us show how to deduce the convergence of the averages \eqref{wienerwinter} from Corollary~\ref{thm:cor1}. Clearly, we may assume that $R(0)=0$. We consider the measure space $(Y,\nu)$ where $Y:=X\times\TT$, $\nu:=\mu\times\lambda$, and $\lambda$ is the normalized Lebesgue measure on $\TT$. We equip the space $(Y,\nu)$ with the family of measure preserving commuting transformations $S_1, S_2,\ldots,S_m,S_{m+1}\colon X\times\TT\to X\times\TT$ where, for $j=1,\ldots,m$, we put $S_j:={\rm Id}\times D_j$  with
\begin{equation*}
    D_j(\xi):=\ex(a_j)\xi
\end{equation*}
being a rotation on $\TT$, and $S_{m+1}:=T\times{\rm Id}.$
We consider the following polynomial mapping
 \begin{equation*}
     \calP(n):=(n^1,n^2,\ldots,n^m,P(n))\colon\ZZ\to\ZZ^{m+1}.
 \end{equation*}
By Corollary~\ref{thm:cor1}, we know that the averages 
\begin{equation}\label{avehwin}
    \frac{1}{2N+1}\sum_{n=-N}^N h\big(S_1^{n^1}\cdots S_m^{n^m}S_{m+1}^{P(n)}y\big),\quad y\in X\times\TT
\end{equation}
converge $\nu$-almost everywhere for any $h\in L^p(Y,\nu)$. If, for $f\in L^p(X,\mu)$, we consider the function $h(y):=f(x)\xi$, then we see that the convergence of the averages \eqref{wienerwinter} follows from the convergence of the averages \eqref{avehwin}. 

The procedure described above can be extended to obtain that, for $\calP$ being a polynomial mapping of the form \eqref{polymap} and $\calR\colon\RR^k\to\RR$ being a polynomial with real coefficients, the averages 
\begin{equation*}
\calA_{t}^{\calP,\calR,k',k''}f(x):=\frac{1}{\vartheta_\Omega(t)}\sum_{(n,p)\in\ZZ^{k'}\times(\pm\PP)^{k''}}f(S_1^{\calP_1(n,p)}\cdots S_d^{\calP_d(n,p)}x)\mathds{1}_{\Omega_t}(n,p)\ex\big(\calR(n,p)\big)\Big(\prod_{i=1}^{k''}\log|p_i|\Big)
\end{equation*}
and
\begin{equation*}
\calH_t^{\calP,\calR,k',k''}f(x):=\sum_{(n,p)\in\ZZ^{k'}\times(\pm\PP)^{k''}}f(S_1^{\calP_1(n,p)}\cdots S_d^{\calP_d(n,p)}x)K(n,p)\mathds{1}_{\Omega_t}(n,p)\ex\big(\calR(n,p)\big)\Big(\prod_{i=1}^{k''}\log|p_i|\Big)
\end{equation*}
converge $\mu$-almost everywhere for any $f\in L^p(X,\mu)$ with $p\in(1,\infty)$. Moreover, we can deduce that the analogue of Corollary~\ref{thm:cor1} holds for $\calA_{t}^{\calP,\calR,k',k''}$ and $\calH_t^{\calP,\calR,k',k''}$. 

Unfortunately, we are not able to prove the Wiener--Wintner theorem for the averages $\calA_{t}^{\calP,\calR,k',k''}$ and $\calH_t^{\calP,\calR,k',k''}$. In our case, that would mean showing that, for any $M\in\NN$, there is a subset of $X$ of full measure on which the convergence holds regardless of the choice of polynomial $\calR$ with $\deg\calR\leq M$.

It is an interesting question whenever the Wiener--Wintner theorem can be somehow deduced from the inequality
\begin{equation}\label{oscilquestww}
    \sup_{N\in\NN}\sup_{I\in\mathfrak{S}_N(\RR_+)}\norm[\big]{O_{I,N}^2(\calA_t^{\calP,\calR,k',k''} f:t>0)}_{L^p(X,\mu)}\le C_{p,d,k,\deg\calP,\deg\calR}\norm{f}_{L^p(X,\mu)}.
\end{equation}
 This question is motivated by the fact that the constant in \eqref{oscilquestww} depends only on the degree of $\calR$ and not its coefficients. We hope to investigate this problem in the near future.

\end{document}